\documentclass[a4paper,11pt]{article}

\usepackage[top=3.0cm, bottom=3.0cm, inner=3.0cm, outer=3.0cm, includefoot]{geometry}

\usepackage{amssymb}
\usepackage{amsmath}
\usepackage{amsthm}
\usepackage{bbm}
\usepackage{relsize}
\usepackage{theoremref}
\usepackage{tikz}
\usepackage{subcaption}
\usepackage{caption}
\usepackage{color,soul}
\usepackage{verbatim}
\usepackage{diagbox}
\usepackage{url}
\usepackage{pgfplots}
\pgfplotsset{compat=1.18}

\usepackage[T1]{fontenc}
\usepackage[utf8]{inputenc}
\usepackage{authblk}

\title{Graphs with Lin-Lu-Yau curvature at least one and regular bone-idle graphs\footnotetext{E-mail address: \texttt{moritz.hehl@uni-leipzig.de}}}

\author{Moritz Hehl}
\affil{Institute of Mathematics, Universit{\"a}t Leipzig, 04109 Leipzig, Germany }
\date{\today}

\theoremstyle{plain}
\newtheorem{lemma}{Lemma}[section]
\newtheorem{theorem}[lemma]{Theorem}

\newtheorem{corollary}[lemma]{Corollary}
\newtheorem{conjecture}[lemma]{Conjecture}

\theoremstyle{definition}
\newtheorem{definition}[lemma]{Definition}

\newtheorem{remark}[lemma]{Remark}

\newtheorem*{theoremLLY}{Theorem \ref{main_result_1}}
\newtheorem*{theoremComparison}{Theorem \ref{kappa_vergleich}}
\newtheorem*{theoremBoneIdle}{Theorem \ref{bone_idle_3_reg}}

\numberwithin{equation}{section}

\begin{document}

\maketitle

\begin{abstract}
    We study the Ollivier-Ricci curvature and its modification introduced by Lin, Lu, and Yau on graphs. We provide a complete characterization of all graphs with Lin-Lu-Yau curvature at least one. We then explore the relationship between the Lin-Lu-Yau curvature and the Ollivier-Ricci curvature with vanishing idleness on regular graphs. An exact formula for the difference between these two curvature notions is established, along with an equality condition. This condition allows us to characterize edges that are bone-idle in regular graphs. Furthermore, we demonstrate the non-existence of 3-regular bone-idle graphs and present a complete characterization of all 4-regular bone-idle graphs. We also show that there exist no 5-regular bone-idle graphs that are symmetric or a Cartesian product of a 3-regular and a 2-regular graph.
    \bigskip

    \bigskip


\end{abstract}

\section{Introduction and statement of results}

Since the introduction of the geometric notion of curvature by Gauss and Riemann over 150 years ago, it has played a central role in differential geometry. Among the different types of curvature, Ricci curvature is of particular importance, serving as a fundamental tool in the study of Riemannian manifolds. Given its importance in differential geometry, it is natural to seek extensions of Ricci curvature to broader classes of metric spaces beyond Riemannian manifolds. This pursuit has led to the development of various generalized curvature notions for non-smooth or discrete structures, see, e.g., Bakry-{\'E}mery \cite{BE1985}, Erbar-Maas \cite{EM2012}, Mielke \cite{M2013} and Forman \cite{F2003}.

In this work, we study a notion of Ricci-curvature introduced by Ollivier in 2009 \cite{Ollivier2009}. Von Renesse and Sturm \cite{RS2005} established a connection between Ricci curvature and optimal transport on smooth Riemannian manifolds. Building upon their results, Ollivier developed a discrete notion of Ricci curvature on metric spaces equipped with Markov chains or a measure, known as the \textit{Ollivier-Ricci curvature}. This approach leverages optimal transport theory, as Ollivier's definition of curvature is based on the Wasserstein distance. 

In the context of locally finite graphs, Ollivier's notion of Ricci curvature has recently received considerable attention. In this setting, the Ollivier-Ricci curvature $\kappa_{\alpha}$ is defined on the edges of the graph and depends on an \textit{idleness parameter} $\alpha \in [0,1]$. Ollivier considered idleness parameter $\alpha = 0$ and $\alpha = \frac{1}{2}$. In 2011, Lin, Lu, and Yau \cite{LLY2011} introduced a modification of the Ollivier-Ricci curvature, by computing the derivative of the curvature with respect to the idleness parameter. We will refer to this modification as \textit{Lin-Lu-Yau curvature} and denote it by $\kappa$. 

Our first result is a complete characterization of the graphs with Lin-Lu-Yau curvature greater than or equal to one for every edge.

\begin{theorem}\thlabel{main_result_1}
    Let $G = (V,E)$ be a locally finite graph. Then $\kappa(x,y) \geq 1$ for every edge $x\sim y\in E$ if and only if the minimum degree $\delta(G) \geq \vert V \vert -2$.
\end{theorem}

We then focus on the relationship between the $0$-Ollivier-Ricci curvature $\kappa_{0}$ and the Lin-Lu-Yau curvature $\kappa$. For regular graphs, we derive an exact formula for the difference between these two curvature notions.

\begin{theorem}\thlabel{kappa_vergleich}
    Let $G=(V,E)$ be a locally finite graph. Let $x,y \in V$ be of equal degree $d$ with $x\sim y$. If $\vert S_{1}(x)\cap S_{1}(y)\vert < d-1$, then 
    \begin{equation*}
        \kappa(x,y) - \kappa_{0}(x,y) = \frac{1}{d}\Biggl(3 - \sup_{\phi \in \mathcal{O}_{xy}} \sup_{z \in S_{1}(x)\setminus B_{1}(y)} d(z, \phi(z))\Biggr),
    \end{equation*}
    where $\mathcal{O}_{xy}$ denotes the set of optimal assignments between $S_{1}(x)\setminus B_{1}(y)$ and $S_{1}(y)\setminus B_{1}(x)$. If $\vert S_{1}(x)\cap S_{1}(y) \vert=d-1$, then 
    \begin{equation*}
        \kappa(x,y) - \kappa_{0}(x,y) = \frac{2}{d}.
    \end{equation*}
\end{theorem}

Finally, we study an analog of Ricci-flat manifolds. Ricci-flat Lorentzian manifolds play an important role in theoretical physics as solutions to Einstein's field equations in a vacuum with vanishing cosmological constant. As an analog, one might consider graphs where the Ollivier-Ricci curvature vanishes everywhere. In this work, we impose an even stronger condition. Namely, that the Ollivier-Ricci curvature $\kappa_{\alpha}$ vanishes everywhere for every idleness parameter $\alpha$. We refer to such graphs as \textit{bone-idle}. It turns out that bone-idleness is equivalent to the vanishing of both the 
$0$-Ollivier-Ricci curvature and the Lin-Lu-Yau curvature $\kappa $ everywhere. Therefore, we can apply our results on the relation between the Ollivier-Ricci curvature and the Lin-Lu-Yau curvature. Using this, we characterize edges that are bone-idle in regular graphs. Furthermore, we show that no 3-regular bone-idle graph exists.

\begin{theorem}\thlabel{bone_idle_3_reg}
    Let $G=(V,E)$ be a locally finite graph. Suppose that $G$ is bone-idle, then $G$ is not $3$-regular.
\end{theorem}

We also provide a complete characterization of 4-regular bone-idle graphs in Section \ref{bone_idle_4_reg} and discuss the existence of 5-regular bone-idle graphs.

We conclude this introduction by providing an outline of the remainder of the paper.  In Section \ref{defs_and_notations} we review the relevant concepts of Graph Theory and Optimal Transport Theory, and introduce the Ollivier-Ricci curvature, as well as its modification by Lin, Lu, and Yau. In Section \ref{lly_geq_1} we provide a complete characterization of all graphs with Lin-Lu-Yau curvature at least one. In Section \ref{relation_curvatures} we explore the relationship between the Lin-Lu-Yau curvature and the $0$-Ollivier-Ricci curvature on regular graphs. Finally, in Section \ref{bone_idleness}, we present our findings on bone-idle graphs.

\section{Definitions and notations}\label{defs_and_notations}

We begin by reviewing some fundamental concepts of Graph Theory and Optimal Transport Theory. We then introduce Ollivier's discrete notion of Ricci curvature on graphs, as well as its modification by Lin, Lu, and Yau.

\subsection{Graph Theory}

A \textit{simple graph} $G=(V,E)$ is an unweighted, undirected graph that contains no multiple edges or self-loops. For two vertices $x,y \in V$ we denote the existence of an edge between $x$ and $y$ by $x \sim y$. For any two vertices $x, y \in V$, the \textit{shortest-path distance} $d(x,y)$ is the number of edges in a shortest path connecting $x$ and $y$. If no such path exists, $d(x,y)$ is defined to be infinity. The \textit{diameter} of $G$ is denoted by $diam(G) = \max_{x,y \in V} d(x,y)$. The \textit{girth} of $G$ is the length of a shortest cycle contained in $G$. If $G$ does not contain any cycles, the girth is defined to be infinity.

For $x \in V$ and $r \in \mathbb{N}$ we define the \textit{$r$-sphere centered at $x$} as $S_{r}(x) = \left\{y \in V: d(x,y) = r\right\}$ and the \textit{$r$-ball centered at $x$} as $B_{r}(x) = \left\{y \in V: d(x,y) \leq r\right\}$. For an edge $x \sim y$, we denote by $N_{xy} = S_{1}(x) \cap S_{1}(y)$ the set of common neighbors of $x$ and $y$.

The \textit{degree} of a vertex $x \in V$ is denoted by $d_{x} = \vert S_{1}(x)\vert$. The minimum degree of a graph $G$ is denoted by $\delta(G) = \min_{x\in V}d_{x}$. The graph $G$ is called \textit{locally finite} if every vertex has finite degree. The graph is said to be \textit{$d$-regular} if every vertex has the same degree $d$.

An \textit{automorphism} of a graph $G=(V,E)$ is a permutation $\sigma$ of the vertex set $V$ such that for any pair of vertices $x,y \in V$, $x\sim y$ if and only if $\sigma(x) \sim \sigma(y)$.  In other words, an automorphism is an isomorphism of $G$ to itself. An \textit{edge automorphism} of $G=(V,E)$ is a permutation of the edge set $E$ that sends edges with a common endpoint into edges with a common endpoint.

A graph is called \textit{vertex-transitive} if, for any two vertices $x,y\in V$, there exists an automorphism $\sigma$ of the graph such that $\sigma(x) = y$. A graph is called \textit{edge-transitive} if, for any two edges $e_{1}, e_{2} \in E$, there exists an edge automorphism $\gamma$ such that $\gamma(e_{1}) = e_{2}$.

We call a graph \textit{symmetric} if it is both vertex-transitive and edge-transitive.

We conclude this section by defining the \textit{Cartesian product} of graphs. Given two graphs $G=(V_{G}, E_{G})$ and $H=(V_{H}, E_{H})$, the Cartesian product $G \square H$, is a graph with vertex set $V_{G} \times V_{H}$. Two vertices $(x_{1},y_{1})$ and $(x_{2},y_{2})$ are adjacent in $G \square H$ if and only if either $x_{1} = x_{2}$ and $y_{1} \sim y_{2} \in E_{H}$, or $x_{1} \sim x_{2} \in E_{G}$ and $y_{1} =y_{2}$.

\subsection{Ollivier-Ricci curvature and its modification}
The Wasserstein distance, a metric defined on the space of probability measures, is a fundamental concept in optimal transport theory.

\begin{definition}[Wasserstein distance]
    Let $G=(V,E)$ be a locally finite graph. Let $\mu_{1}, \mu_{2}$ be two probability measures on $V$. The \textit{Wasserstein distance} between $\mu_{1}$ and $\mu_{2}$ is defined as 
    \begin{equation} \label{eq:1}
        W_{1}(\mu_{1}, \mu_{2}) = \inf_{\pi \in \Pi(\mu_{1}, \mu_{2})} \sum_{x \in V} \sum_{y \in V} d(x,y)  \pi(x,y),
    \end{equation}
    where
    \begin{equation*}
        \Pi(\mu_{1}, \mu_{2}) = \left \{ \pi: V \times V \to [0,1]: \sum_{y \in V} \pi(x,y)=\mu_{1}(x), \; \sum_{x \in V} \pi(x,y)=\mu_{2}(y) \right \}. 
    \end{equation*}
\end{definition}

Intuitively, imagine two distributions given by $\mu_{1}$ and $\mu_{2}$ as piles of earth. The Wasserstein distance measures the minimal effort required to transform one pile of earth into another. We call $\pi \in \Pi(\mu_{1}, \mu_{2})$ a \textit{transport plan} and if the infimum in \ref{eq:1} is attained, we call $\pi$ an \textit{optimal transport plan} transporting $\mu_{1}$ to $\mu_{2}$. 

It is a well-known fact in optimal transport theory that no mass needs to be moved when it is shared between the two probability measures.

\begin{lemma}\thlabel{dontmove}
    Let $G=(V,E)$ be a locally finite graph. Let $\mu_{1}, \mu_{2}$ be two probability measures on $V$. Then there exists an optimal transport plan $\pi$ transporting $\mu_{1}$ to $\mu_{2}$ satisfying
    \begin{equation*}
        \pi(x,x) = \min\{\mu_{1}(x), \mu_{2}(x)\}
    \end{equation*}
    for all $x\in V$.
\end{lemma}

To introduce Ollivier's notion of Ricci curvature on graphs, we define the probability measures $\mu_{x}^{\alpha}$ for $x \in V$ and $\alpha \in [0,1]$ by
\begin{equation*}
    \mu_{x}^{\alpha}(y) =
    \begin{cases}
        \alpha, & \text{if $y = x$;}\\
        \frac{1-\alpha}{d_{x}}, & \text{if $y \sim x$;}\\
        0, & \text{otherwise.}
    \end{cases}
\end{equation*}

Then, the Ollivier-Ricci curvature is defined as follows.

\begin{definition}[Ollivier-Ricci curvature]
    Let $G=(V,E)$ be a locally finite graph. We define the \textit{$\alpha$-Ollivier-Ricci curvature} of an edge $x \sim y$ by
    \begin{equation*}
        \kappa_{\alpha}(x,y) = 1 - W_{1}(\mu_{x}^{\alpha}, \mu_{y}^{\alpha}).
    \end{equation*}
    The parameter $\alpha$ is called the \textit{idleness}.
\end{definition}

We present a more intuitive formula for the Ollivier-Ricci curvature, as provided in \cite{Eidi2020}. Let $\nu_{i}^{\alpha}$ be the mass transported with distance $i$ under an optimal transport plan transporting $\mu_{x}^{\alpha}$ to $\mu_{y}^{\alpha}$. Then 
\begin{equation*}
    \sum_{i=0}^{3} \nu_{i}^{\alpha} = 1 \quad \text{and} \quad W_{1}(\mu_{x}^{\alpha}, \mu_{y}^{\alpha}) = \sum_{i=1}^{3} i \nu_{i}^{\alpha}.
\end{equation*}
Therefore, we obtain 
\begin{equation*}
    \kappa_{\alpha}(x,y) = \nu_{0}^{\alpha} - \nu_{2}^{\alpha} - 2\nu_{3}^{\alpha}.
\end{equation*}

Ollivier considered idleness parameters $\alpha = 0$ and $\alpha = \frac{1}{2}$. In \cite{Bourne2018}, the authors study the Ollivier-Ricci curvature as a function of the idleness parameter. To this end, they introduced the \textit{Ollivier-Ricci idleness function} $\alpha \to \kappa_{\alpha}(x,y)$. It was first shown by Lin, Lu, and Yau in \cite{LLY2011}, that the idleness function is concave. Using that $\kappa_{1}(x,y) = 0$, this implies that the function $h(\alpha) = \frac{\kappa_{\alpha}(x,y)}{1-\alpha}$ is increasing over the interval $[0,1)$. They further showed that $h(\alpha)$ is bounded and thus, the limit $\lim_{\alpha \to 1} h(\alpha)$ exists. Lin, Lu, and Yau used this result to introduce a modified version of the Ollivier-Ricci curvature that does not depend on the idleness.

\begin{definition}[Lin-Lu-Yau curvature]
    Let $G=(V,E)$ be a locally finite graph. The \textit{Lin-Lu-Yau curvature} of an edge $x \sim y$ is defined as 
    \begin{equation*}
        \kappa(x,y) = \lim_{\alpha \to 1} \frac{\kappa_{\alpha}(x,y)}{1-\alpha}.
    \end{equation*}
\end{definition}

\begin{remark}
    Observe that $\kappa_{1}(x,y) = 0$ for any edge $x\sim y$. Thus, $\kappa(x,y)$ is the negative of the derivative of $\kappa_{\alpha}(x,y)$ with respect to the idleness parameter in $\alpha=1$.
\end{remark}

In what follows, we write $Ric(G) \geq k$ ($Ric(G) = k$) if $\kappa(x,y) \geq k$ ($\kappa(x,y) = k$) for all edges $x\sim y$ in $G$.

Bourne et al. \cite{Bourne2018} showed that the idleness function is piecewise linear with at most 3 linear parts. They also derived the length of the last linear part.

\begin{theorem}[\cite{Bourne2018}, Theorem 4.4]
    Let $G=(V,E)$ be a locally finite graph and let $x,y \in V$ with $x \sim y$ and $d_{x} \geq d_{y}$. Then $\alpha \to \kappa_{\alpha}(x,y)$ is linear over $\left[\frac{1}{d_{x} + 1}, 1\right].$
\end{theorem}

Thus, we obtain the following relation between the $\alpha$-Ollivier-Ricci curvature and its modification by Lin, Lu, and Yau as an immediate consequence of the mean value theorem.

\begin{theorem}\thlabel{relation_curvature}
    Let $G=(V,E)$ be a locally finite graph and let $x,y \in V$ with $x \sim y$ and $d_{x} \geq d_{y}$. Then
    \begin{equation*}
        \kappa_{\alpha}(x,y) = (1-\alpha)\kappa(x,y)
    \end{equation*}
    for $\alpha \in \left[\frac{1}{d_{x} +1},1\right]$.
\end{theorem}

Hence, the Lin-Lu-Yau curvature coincides up to a scaling factor with the $\alpha$-Ollivier-Ricci curvature for large values of $\alpha$. 

For regular graphs, the optimal transport problem reduces to an optimal assignment problem between subsets of the $1$-spheres.  In \cite{MH2024}, the author uses this observation to derive a simplified formula for the Lin-Lu-Yau curvature. To formalize this approach, we introduce the concept of an optimal assignment.

\begin{definition}[Optimal assignment]
    Let $G=(V,E)$ be a locally finite graph. Let $x,y \in V$ be of equal degree $d$ with $x \sim y$. We call a bijection $\phi: S_{1}(x)\setminus B_{1}(y) \to S_{1}(y)\setminus B_{1}(x)$ an \textit{assignment} between $S_{1}(x)\setminus B_{1}(y)$ and $S_{1}(y)\setminus B_{1}(x)$. Denote by $\mathcal{A}_{xy}$ the set of all such assignments. We call $\phi \in \mathcal{A}_{xy}$ an \textit{optimal assignment} if 
    \begin{equation*}
        \sum_{z \in S_{1}(x)\setminus B_{1}(y)} d(z,\phi(z)) = \inf_{\psi \in \mathcal{A}_{xy}} \sum_{z \in S_{1}(x)\setminus B_{1}(y)} d(z,\psi(z)).
    \end{equation*}
    The set of all optimal assignments between \( S_{1}(x) \setminus B_{1}(y) \) and \( S_{1}(y) \setminus B_{1}(x) \) is denoted by \( \mathcal{O}_{xy} \).
\end{definition}

\begin{remark}
    Observe that the condition that $x$ and $y$ have the same degree is necessary to ensure that $\vert S_{1}(x)\setminus B_{1}(y) \vert = \vert S_{1}(y)\setminus B_{1}(x) \vert$.
\end{remark}

\begin{theorem}[\cite{MH2024}, Theorem 4.3]\thlabel{Lin_Lu_Yau_curvature}
    Let $G=(V,E)$ be a locally finite graph. Let $x,y \in V$ be of equal degree $d$ with $x \sim y$. Then the Lin-Lu-Yau curvature
    \begin{equation*}
        \kappa(x,y) = \frac{1}{d}\Biggl(d+1 - \inf_{\phi \in \mathcal{A}_{xy}} \mathlarger{\sum}_{z \in S_{1}(x) \setminus B_{1}(y)}d(z, \phi(z)) \Biggr),
    \end{equation*}
\end{theorem}

A similar formula holds true for the Ollivier-Ricci curvature in the case of vanishing idleness, i.e., $\alpha = 0$.

\begin{theorem}[\cite{MH2024}, Theorem 4.8]\thlabel{kappa_null}
    Let $G=(V,E)$ be a locally finite graph. Let $x,y \in V$ be of equal degree $d$ with $x\sim y$. Then 
    \begin{equation*}
        \kappa_{0}(x,y) = \frac{1}{d}\Biggl(d - \inf_{\phi} \mathlarger{\sum}_{z \in S_{1}(x)\setminus S_{1}(y)}d(z, \phi(z))\Biggr),
    \end{equation*}
    where the infimum is taken over all bijections $\phi$ between $S_{1}(x)\setminus S_{1}(y)$ and $S_{1}(y)\setminus S_{1}(x)$.
\end{theorem}

\section{Lin-Lu-Yau curvature at least one}\label{lly_geq_1}

In this section, we characterize all graphs for which the Lin-Lu-Yau curvature is greater than or equal to one. To this end, we first establish the following upper bound on the Lin-Lu-Yau curvature.

\begin{theorem}\thlabel{upper_bound}
    Let $G=(V,E)$ be a locally finite graph and let $x,y \in V$ with $x \sim y$ and $d_{x} \geq d_{y}$. Then
    \begin{equation*}
        \kappa(x,y) \leq \frac{\vert N_{xy} \vert + 2}{d_{x}}.
    \end{equation*}
\end{theorem}

\begin{remark}
    In \cite[Theorem 4]{JL2013}, the authors establish a similar upper bound for the $0$-Ollivier-Ricci curvature.
\end{remark}

\begin{proof}
    If $d_{x} = d_{y}$, observe that $d(i,j) \geq 1$ for every $i \in S_{1}(x) \setminus B_{1}(y)$ and $j \in S_{1}(y) \setminus B_{1}(x)$. Using \thref{Lin_Lu_Yau_curvature}, we obtain
    \begin{align*}
        \kappa(x,y) &= \frac{1}{d_{x}}\Biggl(d_{x}+1 - \inf_{\phi \in \mathcal{A}_{xy}} \mathlarger{\sum}_{z \in S_{1}(x) \setminus B_{1}(y)}d(z, \phi(z)) \Biggr) \\
        &\leq \frac{1}{d_{x}}\Biggl(d_{x}+1 - \vert S_{1}(x) \setminus B_{1}(y) \vert\Biggr) \\
        &= \frac{\vert N_{xy} \vert + 2}{d_{x}}.
    \end{align*}

    Now, assume $d_{x} > d_{y}$ and set $\alpha = \frac{1}{d_{y}+1}$. According to \thref{dontmove}, there exists an optimal transport plan $\pi$ transporting $\mu_{x}^{\alpha}$ to $\mu_{y}^{\alpha}$, such that $\pi(j,j) = \min\{\mu_{x}^{\alpha}(j), \mu_{y}^{\alpha}(j)\}$ for all $j \in V$. Therefore,  
    \begin{equation*}
        \nu_{0}^{\alpha} = (\vert N_{xy} \vert + 1) \frac{1-\alpha}{d_{x}} + \frac{1-\alpha}{d_{y}}
    \end{equation*}
    where $\nu_{i}^{\alpha}$ denotes the mass transported with distance $i$ under $\pi$. 

    Observe that the set $I = \{z \in B_{1}(x)\setminus\{y\}: \pi(z,y)>0\}$ is non-empty, as $\mu_{y}^{\alpha}(y) - \pi(y,y) > 0$. Furthermore, $I$ does not contain $x$ and common neighbors of $x$ and $y$. Therefore, every vertex in $I$ is at distance two of $y$. Hence,
    \begin{equation*}
        \nu_{2}^{\alpha} \geq \mu_{y}^{\alpha}(y) - \pi(y,y) = \alpha - \frac{1-\alpha}{d_{x}}.
    \end{equation*}
    Using $\alpha = \frac{1-\alpha}{d_{y}}$, we conclude that
    \begin{align*}
        \kappa_{\alpha}(x,y) &= \nu_{0}^{\alpha} -  \nu_{2 }^{\alpha} - 2 \nu_{3}^{\alpha} \\
        &\leq \frac{\vert N_{xy} \vert + 2}{d_{x}}(1-\alpha). 
    \end{align*}
    Finally, we apply \thref{relation_curvature}, and obtain
    \begin{equation*}
        \kappa(x,y) = \frac{1}{1-\alpha}\kappa_{\alpha}(x,y) \leq \frac{\vert N_{xy} \vert + 2}{d_{x}}. 
    \end{equation*}
\end{proof}

The following Bonnet-Myers type theorem on graphs will be of importance.

\begin{theorem}[Discrete Bonnet-Myers Theorem \cite{LLY2011}]\thlabel{Bonnet_Myers}
    Let $G=(V,E)$ be a locally finite graph. If $Ric(G) \geq k > 0$, then the diameter of the graph $G$ is bounded as follows:
    \begin{equation*}
        diam(G) \leq \frac{2}{k}.
    \end{equation*}
\end{theorem}

We are now prepared to prove the main result of this section.

\begin{theoremLLY}
    Let $G=(V,E)$ be a locally finite graph. Then $Ric(G) \geq 1$ if and only if $\delta(G) \geq \vert V \vert -2$.
\end{theoremLLY}

\begin{proof}
    "$\impliedby$" Let $x \sim y$ be an arbitrary edge in $G$ and assume, without loss of generality, that $d_{x} \geq d_{y}$. As $\delta(G) \geq \vert V \vert -2$, we have $V = S_{1}(x) \cup S_{1}(y)$. Thus,
    \begin{equation*}
        \vert V \vert = d_{x} + d_{y} - \vert N_{xy} \vert.
    \end{equation*}
    Using $d_{y} \geq \vert V \vert -2$, we obtain $\vert N_{xy} \vert \geq d_{x} -2$. If $\vert N_{xy} \vert = d_{x}-1$, then $d_{x} = d_{y}$ must hold and we can apply \thref{Lin_Lu_Yau_curvature}, yielding $\kappa(x,y) = \frac{d+1}{d} > 1$. 
    
    If $\vert N_{xy} \vert = d_{x} - 2$, then $x$ has exactly one neighbor $z$ that is not adjacent to $y$. As $\delta(G) \geq \vert V \vert -2$, $z$ must be adjacent to every vertex in $G$ besides $y$. Now, assume $\alpha = \frac{1}{d_{y} +1}$ and let $\pi$ be an optimal transport plan transporting $\mu_{x}^{\alpha}$ to  $\mu_{y}^{\alpha}$, satisfying the property stated in \thref{dontmove}. As $z$ is adjacent to every vertex in $G$ besides $y$, we obtain 
    \begin{align*}
        \nu_{0}^{\alpha} &= (\vert N_{xy} \vert + 1) \frac{1-\alpha}{d_{x}} + \frac{1-\alpha}{d_{y}}, \\
        \nu_{2}^{\alpha} &= \frac{1-\alpha}{d_{y}} - \frac{1-\alpha}{d_{x}}, \\
        \nu_{3}^{\alpha} &= 0.
    \end{align*}
    Therefore, 
    \begin{equation*}
        \kappa_{\alpha}(x,y) =  \frac{\vert N_{xy} \vert + 2}{d_{x}}(1-\alpha) = (1-\alpha).
    \end{equation*}
    Using \thref{relation_curvature}, we conclude $\kappa(x,y) = 1$.

    "$\implies$" We show this by contradiction. Let $G$ be a graph with $\delta(G) < \vert V \vert -2$ and $Ric(G) \geq 1$. Let $x\in V$ such that $d_{x} = \delta(G)$ and let $y$ be an arbitrary neighbor of $x$. Observe that $\vert N_{xy} \vert \leq d_{x} -1$. If $d_{y}>d_{x}+1$, then 
    \begin{equation*}
        \kappa(x,y) \leq \frac{\vert N_{xy} \vert+2}{d_{y}} \leq \frac{d_{x} +1}{d_{y}} < 1,
    \end{equation*}
    contradicting our assumption. 

    Therefore, we now assume $d_{x} \leq d_{y} \leq d_{x}+1$. As
    \begin{equation*}
        1 \leq \kappa(x,y) \leq \frac{\vert N_{xy} \vert + 2}{d_{y}},
    \end{equation*}
    we conclude $\vert N_{xy} \vert \geq d_{y} - 2 \geq d_{x} - 2$. First, assume $\vert N_{xy} \vert = d_{x}-1$. As $d_{x} < \vert V \vert -2$, there exist two vertices $i,j \in V$ such that $x \not\sim i,j$. According to \thref{Bonnet_Myers}, $d(i,x) = 2$ and $d(j,x) = 2$ must hold true. If both $i$ and $j$ are adjacent to $y$, then $d_{y} \geq d_{x} +2$, which contradicts our assumption. Thus, there exists a $z\in N_{xy}$ such that, without loss of generality, $i \sim z$ and $i\not\sim y$, $i \not\sim x$. Therefore, $\vert N_{iz} \vert \leq d_{z} - 3$, which leads to
    \begin{equation*}
        \kappa(i,z) \leq \frac{\vert N_{iz} \vert + 2}{\max\{d_{i}, d_{z}\}} \leq \frac{d_{z} - 1}{\max\{d_{i}, d_{z}\}} < 1, 
    \end{equation*}
    contradicting $Ric(G) \geq 1$.

    Finally, assume $\vert N_{xy} \vert = d_{x}-2$. Observe that, in this case, it is necessary for the equality $d_{y} = d_{x}$ to hold. 
    Denote the vertex in $S_{1}(x)\setminus B_{1}(y)$ by $z_{1}$ and the vertex in $S_{1}(y)\setminus B_{1}(x)$ by $z_{2}$. According to \thref{Lin_Lu_Yau_curvature}, we have
    \begin{equation*}
        \kappa(x,y) = \frac{1}{d} \Big(d + 1 - d(z_{1}, z_{2})\Big).
    \end{equation*}
    Using $\kappa(x,y) \geq 1$, we conclude $d(z_{1}, z_{2})=1$.
    As $d_{x} < \vert V \vert -2$, there exist two vertices $i,j \in V$ such that $x \not\sim i,j$. It is not possible for $y$ to be adjacent to both $i$ and $j$, as this would contradict $d_{y} = d_{x}$. Therefore, without loss of generality, we can assume $y \not\sim i$. Again, according to \thref{Bonnet_Myers}, we must have $d(x,i)=2$. Thus, it exists a $z\in S_{1}(x)\setminus\{y\}$ such that $z \sim i$. If $z = z_{1}$, then $\vert N_{xz} \vert \leq d_{z} -3$, because $x \not\sim z_{2}$ and $x \not\sim i$. Therefore, we conclude
    \begin{equation*}
        \kappa(x,z) \leq \frac{\vert N_{xz} \vert + 2}{d_{z}} < 1,
    \end{equation*}
    contradicting out assumption. Hence, $z \in N_{xy}$ must hold. As before, we obtain $\vert N_{iz} \vert \leq d_{z} - 3$, because $i \not\sim x$ and $i \not\sim y$, which leads to 
    \begin{equation*}
        \kappa(i,z) \leq \frac{\vert N_{iz} \vert + 2}{\max\{d_{i}, d_{z}\}} \leq \frac{d_{z} - 1}{\max\{d_{i}, d_{z}\}} < 1. 
    \end{equation*}
    This contradiction concludes the proof.
\end{proof}

\begin{corollary}\thlabel{regular_lly_one}
    Let $G=(V,E)$ be a regular graph with $Ric(G) \geq 1$. Then $G$ is isomorphic to one of the following graphs:
    \begin{itemize}
        \item[$(i)$] A complete graph, satisfying $Ric(G) = \frac{\vert V \vert}{\vert V \vert -1}$,
        \item[$(i)$] a cocktail party graph, satisfying $Ric(G) = 1$.
    \end{itemize}
\end{corollary}

\begin{remark}
    A cocktail party graph $G=(V,E)$ is a regular graph of degree $d$, where $d= \vert V \vert -2$.
\end{remark}

\begin{corollary}
    Let $G=(V,E)$ be a locally finite graph with $Ric(G) = 1$. If $\vert V \vert$ is even, then $G$ is a cocktail party graph. If $\vert V \vert$ is odd, then $G$ is the graph with degree sequence $(\vert V \vert - 1, \vert V \vert -2, \dots, \vert V \vert -2)$.
\end{corollary}

\begin{proof}
    Assume $G$ is a locally finite graph, satisfying $Ric(G) = 1$. If $G$ is regular, then $G$ is a cocktail party graph, according to \thref{regular_lly_one}. Note that in this case, $\vert V \vert$ must be even.

    Next, assume $G$ is not regular. According to \thref{main_result_1}, we have $\delta(G) \geq \vert V \vert -2$. If there exist two vertices $x,y$ with degree $d > \vert V \vert -2$, then $\vert N_{xy} \vert = d - 1$, leading to $\kappa(x,y) = \frac{d+1}{d} > 1$, contradicting $Ric(G) = 1$. Therefore, $G$ can only have degree sequence $(\vert V \vert - 1, \vert V \vert -2, \dots, \vert V \vert -2)$. It remains to show that the graph with this degree sequence indeed satisfies $Ric(G) = 1$. To this end, let $x \sim y$ be an arbitrary edge in $G$. If $d_{x} = \vert V \vert - 1 > d_{y} = \vert V \vert - 2$, then $\vert N_{xy} \vert = d_{y} - 1$. This leads to 
    \begin{equation*}
        \kappa(x,y) \leq \frac{\vert N_{xy} \vert + 2}{d_{x}} = 1.
    \end{equation*}
    According to \thref{main_result_1}, we also have $\kappa(x,y) \geq 1$, and therefore $\kappa(x,y) = 1$. 

    Finally, assume $d_{x} = d_{y} = \vert V \vert - 2$. If $\vert N_{xy} \vert = d_{x} -1$, then 
    \begin{equation*}
        \vert S_{1}(x) \cup S_{1}(y) \vert = d_{x} + 1 = \vert V \vert - 1.
    \end{equation*}
    Hence, there exists a vertex that is not adjacent to both $x$ and $y$, resulting in a degree less than $\vert V \vert -2$, which leads to a contradiction. Therefore, $\vert N_{xy} \vert = d_{x} - 2$ must hold, implying that $\kappa(x,y) \leq 1$. According to \thref{main_result_1}, we also have $\kappa(x,y) \geq 1$, and therefore we conclude $\kappa(x,y) =1$. Finally, note that the graph with this degree sequence must contain an odd number of vertices. This can be seen by observing that each vertex of degree $\vert V \vert -2$ is non-adjacent to exactly one other vertex of degree $\vert V \vert -2$.
\end{proof}

We also provide a complete characterization of all graphs with Lin-Lu-Yau curvature greater than one, a result previously established in \cite{Bonini2019}.

\begin{corollary}
    Let $G=(V,E)$ be a locally finite graph with $Ric(G) > 1$. Then $G$ is a complete.
\end{corollary}

\begin{proof}
    According to \thref{main_result_1}, $\delta(G) \geq \vert V \vert -2$ must hold true. If every vertex is of degree $\vert V \vert -2$, then $G$ is a cocktail party graph and therefore $Ric(G) =1$. Thus, there exists at least one vertex $x$ of degree $\vert V \vert -1$. Assume there exists a vertex $y$ of degree $\vert V \vert -2$. Then
    \begin{equation*}
        \kappa(x,y) \leq \frac{\vert N_{xy} \vert +2}{d_{x}} \leq 1,
    \end{equation*}
    contradicting $Ric(G) > 1$. Therefore, $G$ must be the complete graph.
    
\end{proof}

\section{Relation between the two curvature notions}\label{relation_curvatures}

In this section, we derive an exact expression for the difference between the Lin-Lu-Yau curvature and the $0$-Ollivier-Ricci curvature on regular graphs. We begin with the following Lemma, which addresses certain assumptions that can be imposed on an optimal assignment.

\begin{lemma}\thlabel{assumptions_optimal_matching}
    Let $G=(V,E)$ be a locally finite graph. Let $x,y \in V$ be of equal degree d with $x \sim y$.
    Then there exists an optimal assignment $\phi$ between $S_{1}(x)$ and $S_{1}(y)$, such that $\phi(i) = i$ for all $i \in N_{xy}$.

    Furthermore, if $\vert N_{xy} \vert < d-1$, there exists an optimal assignment $\phi$, satisfying the aforementioned property and $\phi(y) \neq x$.
\end{lemma}

\begin{proof}
    The first part of the Lemma is an immediate consequence of the triangle inequality.
    
    For the second part of the proof, assume that $\vert N_{xy} \vert < d-1$ and assume that $\phi(y) = x$. Choose an arbitrary $i \in S_{1}(x) \setminus B_{1}(y)$ and define a new assignment $\phi'$ between $S_{1}(x)$ and $S_{1}(y)$ by
    \begin{equation*}
        \phi'(z) = \begin{cases}
            \phi(i), & \text{if $z = y$;}\\
            x, & \text{if $z = i$;}\\
            \phi(z), & \text{otherwise.}
        \end{cases}
    \end{equation*}
    Since $i \not\in N_{xy}$, we have $d(i, \phi(i)) \geq 1$, leading to
    \begin{equation*}
        d(y, \phi'(y)) + d(i, \phi'(i)) = 2 \leq d(y, \phi(y)) + d(i, \phi(i)). 
    \end{equation*}
    Hence, the new assignment $\phi'$ is still optimal. Finally, note that $\phi'(i) = \phi(i) = i$ for all $i \in N_{xy}$. 
    This concludes the proof.
\end{proof}

We are now ready to prove the main theorem of this section.

\begin{theoremComparison}
    Let $G=(V,E)$ be a locally finite graph. Let $x,y \in V$ be of equal degree $d$ with $x\sim y$. If $\vert N_{xy}\vert < d-1$, then 
    \begin{equation*}
        \kappa(x,y) - \kappa_{0}(x,y) = \frac{1}{d}\Biggl(3 - \sup_{\phi \in \mathcal{O}_{xy}} \sup_{z \in S_{1}(x)\setminus B_{1}(y)} d(z, \phi(z))\Biggr).
    \end{equation*}
    If $\vert N_{xy} \vert=d-1$, then 
    \begin{equation*}
        \kappa(x,y) - \kappa_{0}(x,y) = \frac{2}{d}.
    \end{equation*}
\end{theoremComparison}

\begin{proof}
    \emph{Case 1: $\vert N_{xy} \vert = d-1$.} According to \thref{Lin_Lu_Yau_curvature}, we have $\kappa(x,y) = \frac{d+1}{d}$, while \thref{kappa_null} shows that $\kappa_{0}(x,y) = \frac{d-1}{d}$. Hence,
    \begin{equation*}
        \kappa_{0}(x,y) = \kappa(x,y) - \frac{2}{d}.
    \end{equation*}

    \emph{Case 2: $\vert N_{xy} \vert < d-1$.} Choose $\phi \in \mathcal{O}_{xy}$ and $j \in S_{1}(x) \setminus B_{1}(y)$ such that
    \begin{equation*}
        d(j, \phi(j)) = \sup_{\psi \in \mathcal{O}_{xy}} \sup_{z \in S_{1}(x) \setminus B_{1}(y)} d(z, \psi(z)).
    \end{equation*}
    Next, we define an assignment $\psi$ between $S_{1}(x)\setminus S_{1}(y)$ and $S_{1}(y)\setminus S_{1}(x)$ by
    \begin{equation*}
        \psi(z) = \begin{cases}
            x, & \text{if $z = j$;}\\
            \phi(j), & \text{if $z = y$;} \\
            \phi(z), & \text{otherwise.}
        \end{cases}
    \end{equation*}
    We claim that $\psi$ is an optimal assignment between $S_{1}(x)\setminus S_{1}(y)$ and $S_{1}(y)\setminus S_{1}(x)$. Assume this is not the case. Let $\psi'$ be an optimal assignment. According to \thref{assumptions_optimal_matching}, we can assume that $\psi'(y) \neq x$. As $\psi$ is not optimal,
    \begin{equation}\label{eq:2}
        \sum_{z \in S_{1}(x)\setminus S_{1}(y)} d(z, \psi'(z)) < \sum_{z \in S_{1}(x)\setminus S_{1}(y)} d(z, \psi(z))
    \end{equation}
    must hold. Using $d(j,\psi(j)) = d(y,\psi(y)) = 1$, we obtain
    \begin{equation}\label{eq:3}
        \sum_{z \in S_{1}(x)\setminus S_{1}(y)} d(z, \psi(z)) = \sum_{\substack{z \in S_{1}(x)\setminus B_{1}(y), \\ z \not= j}}d(z,\phi(z)) + 2.
    \end{equation}
    On the other hand, by assumption, we have $\psi'(y) \neq x$. Denote by $k$ the preimage of $x$ under $\psi'$. Then 
    \begin{equation}\label{eq:4}
        \sum_{z \in S_{1}(x)\setminus S_{1}(y)} d(z, \psi'(z)) = \sum_{\substack{z \in S_{1}(x) \setminus B_{1}(y), \\z \not= k }} d(z,\psi'(z)) + 2.
    \end{equation}
    Combining equations \ref{eq:2}, \ref{eq:3} and \ref{eq:4} leads to
    \begin{equation}\label{eq:5}
        \sum_{\substack{z \in S_{1}(x) \setminus B_{1}(y), \\z \not= k }} d(z,\psi'(z)) < \sum_{\substack{z \in S_{1}(x)\setminus B_{1}(y), \\ z \not= j}}d(z,\phi(z)).
    \end{equation}
    Next, we define an assignment $\phi' \in \mathcal{A}_{xy}$ between $S_{1}(x)\setminus B_{1}(y)$ and $S_{1}(y)\setminus B_{1}(x)$ by
    \begin{equation*}
        \phi'(z) = \begin{cases}
            \psi'(y), & \text{if $z = k$;}\\
            \psi'(z), & \text{if $z \in S_{1}(x)\setminus B_{1}(y)$ and $z\not=k$ .}
        \end{cases}
    \end{equation*}
    Due to the optimality of $\phi$, we have
    \begin{equation*}
    \sum_{z \in S_{1}(x) \setminus B_{1}(y)} d(z, \phi'(z)) \geq \sum_{z \in S_{1}(x) \setminus B_{1}(y)} d(z, \phi(z)). 
    \end{equation*}

    We now distinguish the following two cases:

    \emph{Case 1:} $\sum_{z \in S_{1}(x) \setminus B_{1}(y)} d(z, \phi'(z)) = \sum_{z \in S_{1}(x) \setminus B_{1}(y)} d(z, \phi(z))$. In this case, $\phi' \in \mathcal{O}_{xy}$ and by our choice of $\phi$ and $j$, we have
    \begin{equation*}
        d(k, \phi'(k)) \leq d(j, \phi(j)),
    \end{equation*}
    leading to
    \begin{equation*}
        \sum_{\substack{z \in S_{1}(x) \setminus B_{1}(y), \\z \not= k }} d(z,\phi'(z)) \geq \sum_{\substack{z \in S_{1}(x)\setminus B_{1}(y), \\ z \not= j}}d(z,\phi(z)).
    \end{equation*}
    contradicting equation \ref{eq:5}.

    \emph{Case 2:} $\sum_{z \in S_{1}(x) \setminus B_{1}(y)} d(z, \phi'(z)) > \sum_{z \in S_{1}(x) \setminus B_{1}(y)} d(z, \phi(z))$.
    Due to equation \ref{eq:5}, we have 
    \begin{equation*}
        2 + d(j,\phi(j)) \leq d(k, \phi'(k))).
    \end{equation*} 
    As $d(k, \phi'(k)) \leq 3$ and $1 \leq d(j, \phi(j))$, we must have $d(j, \phi(j))= 1$ and therefore 
    \begin{equation*}
        1 = d(j,\phi(j)) \geq d(z,\phi(z)) \geq 1
    \end{equation*} 
    holds for any $z \in S_{1}(x)\setminus B_{1}(y)$, by the choice of $j$ and $\phi$. This contradicts equation \ref{eq:5}, as 
    \begin{equation*}
        d(z, \psi'(z)) \geq 1, \; \forall z \in S_{1}(x)\setminus B_{1}(y).
    \end{equation*}
    Therefore our assumption was wrong and $\psi$ is an optimal assignment between $S_{1}(x)\setminus S_{1}(y)$ and $S_{1}(y)\setminus S_{1}(x)$. Using \thref{kappa_null}, we obtain
    \begin{align*}
        \kappa_{0}(x,y) &= \frac{1}{d}\Biggl(d-\sum_{z \in S_{1}(x)\setminus S_{1}(y)} d(z, \psi(z))\Biggr) \\
        & = \frac{1}{d}\Biggl(d - \sum_{\substack{z \in S_{1}(x)\setminus B_{1}(y), \\ z \not= j}}d(z,\phi(z)) - 2\Biggr) \\
        & = \frac{1}{d}\Biggl(d +1 - \sum_{z \in S_{1}(x)\setminus B_{1}(y)}d(z,\phi(z)) \Biggr) - \frac{1}{d}\Biggl(3 - d(j, \phi(j))\Biggr),
    \end{align*}
    where we used equation \ref{eq:3} for the second equality. Using the optimality of $\phi$ and \thref{Lin_Lu_Yau_curvature}, we obtain
    \begin{equation*}
        \kappa_{0}(x,y) = \kappa(x,y) - \frac{1}{d}\Biggl(3 - d(j, \phi(j))\Biggr).
    \end{equation*}
    The choice of $\phi$ and $j$ concludes the proof.
\end{proof}

The following result was previously established in \cite{Bourne2018} and is an immediate consequence of \thref{kappa_vergleich}.

\begin{corollary}\thlabel{vergleich_einfach}
    Let $G=(V,E)$ be a locally finite graph. Let $x,y\in V$ be of equal degree $d$ with $x \sim y$. Then,
    \begin{equation*}
        \kappa_{0}(x,y) = \kappa(x,y) - \frac{c}{d},
    \end{equation*}
    where $c \in \{0,1,2\}$.
\end{corollary} 

The following result provides a necessary and sufficient condition for $\kappa(x,y) = \kappa_{0}(x,y)$ on edges with endpoints of equal degree.

\begin{corollary}\thlabel{equality_condition}
    Let $G=(V,E)$ be a locally finite graph. Let $x,y \in V$ be of equal degree $d$ with $x \sim y$. Then 
    \begin{equation*}
        \kappa_{0}(x,y) = \kappa(x,y)
    \end{equation*} 
    if and only if there exists an optimal assignment $\phi \in \mathcal{O}_{xy}$ between $S_{1}(x) \setminus B_{1}(y)$ and $S_{1}(y) \setminus B_{1}(x)$ such that
    \begin{equation*}
        \exists z \in S_{1}(x) \setminus B_{1}(y): \; d(z, \phi(z)) = 3.
    \end{equation*}
\end{corollary}

\begin{remark}
    Therefore, in a regular graph with diameter at most two, every edge $x \sim y$ satisfies $\kappa(x,y) > \kappa_{0}(x,y)$. One class of such graphs are $d$-regular graphs with $d \geq \frac{\vert V \vert -1}{2}$.
\end{remark}

Next, we use \thref{equality_condition} to determine an interval where $\kappa$ and $\kappa_{0}$ always coincide.

\begin{corollary}
    Let $G=(V,E)$ be a locally finite graph. Let $x,y \in V$ be of equal degree $d$ with $x \sim y$. If 
    \begin{equation*}
        \kappa(x,y) < -1 + \frac{2\vert N_{xy} \vert+3}{d}.,
    \end{equation*}
    then $\kappa(x,y) = \kappa_{0}(x,y)$.
\end{corollary}

\begin{proof}
    We argue by contradiction. Assume $\kappa(x,y) \neq \kappa_{0}(x,y)$ and 
    \begin{equation*}
        \kappa(x,y) < -1 + \frac{2\vert N_{xy} \vert+3}{d}.
    \end{equation*}
    Let $\phi \in \mathcal{O}_{xy}$ be an optimal assignment between $S_{1}(x)\setminus B_{1}(y)$ and $S_{1}(y)\setminus B_{1}(x)$.
    According to \thref{equality_condition}, we have $d(z, \phi(z)) \leq 2$ for any $z \in S_{1}(x)\setminus B_{1}(y)$. Hence,
    \begin{align*}
        \kappa(x,y) &= \frac{1}{d} \Biggl(d+1 - \sum_{z \in S_{1}(x)\setminus B_{1}(y)}d(z, \phi(z))\Biggr) \\
                    &\geq \frac{1}{d} \Biggl(d + 1 - 2 \vert S_{1}(x)\setminus B_{1}(y)\vert\Biggr) \\
                    &= \frac{1}{d} \Biggl(-d + 3 + 2 \vert N_{xy} \vert\Biggr),
    \end{align*}
    contradicting our assumption.
\end{proof}

\section{Bone-idleness}\label{bone_idleness}

In this section, we examine edges with Ollivier-Ricci curvature equal to zero for every idleness parameter $\alpha$. This concept was originally introduced by Bourne et al. in \cite{Bourne2018} and referred to as \textit{bone-idle}.

\begin{definition}[Bone-idle]
    Let $G=(V,E)$ be a locally finite graph. We say an edge $x \sim y$ is \textit{bone-idle} if $\kappa_{\alpha}(x,y) = 0$ for every $\alpha \in [0,1]$. We say that $G$ is \textit{bone-idle} if every edge of $G$ is bone-idle.
\end{definition}

The notion of Ricci-flatness is weaker but strongly related to bone-idleness.

\begin{definition}[Ricci-flat]
    Let $G=(V,E)$ be a locally finite graph. We call $G$ \textit{Ricci-flat} if $\kappa(x,y) = 0$ for every edge $x\sim y\in E$. We call $G$ \textit{$\alpha$-Ricci flat} if $\kappa_{\alpha}(x,y) = 0$ for every edge $x\sim y\in E$.
\end{definition}

Due to the following Lemma, bone-idleness and the various notions of Ricci-flatness are closely related.

\begin{lemma}\thlabel{bone-idle_vs_ricci-flat}
    Let $G=(V,E)$ be a locally finite graph. Let $x,y \in V$ with $x\sim y$. Then the following are equivalent:
    \begin{enumerate}
        \item[$(i)$] $\kappa_{\alpha}(x,y) = 0$ for all $\alpha \in [0,1]$.
        \item[$(ii)$] $\kappa_{0}(x,y) = \kappa(x,y) = 0$.
    \end{enumerate}
\end{lemma}
    
\begin{proof}
    \emph{(ii)$\implies$(i)}. Let $\alpha \in (0,1)$ be arbitrary. Assume $\kappa_{0}(x,y) = \kappa(x,y) = 0$. Recall that the idleness function is concave, and also note that $\kappa_{1}(x,y)=0$. Hence,
    \begin{equation*}
        \kappa_{\alpha}(x,y) \geq \alpha \kappa_{1}(x,y) + (1-\alpha) \kappa_{0}(x,y) = 0.
    \end{equation*}
    The other inequality follows from the fact that the graph of a concave function lies below its tangent line at each point and that $\kappa_{1}' = -\kappa$:
    \begin{equation*}
        \kappa_{\alpha}(x,y) \leq \kappa_{1}(x,y) + \kappa_{1}'(x,y)\cdot(\alpha - 1) = \kappa(x,y) \cdot (1-\alpha) =  0.
    \end{equation*}
    
    \emph{(i) $\implies$ (ii)}. This is an immediate consequence of the definition of $\kappa$.
\end{proof}

Therefore, a graph is bone-idle if and only if it is Ricci-flat and $0$-Ricci-flat.
Previous works have addressed the classification of Ricci-flat graphs. Cushing et al. \cite{Cushing2018} classified all Ricci-flat graphs with girth at least five.

\begin{theorem}[\cite{Cushing2018}, Theorem 1]\thlabel{ricci_flat}
    Let $G=(V,E)$ be a locally finite graph with girth at least five. Suppose that $G$ is Ricci-flat. Then $G$ is isomorphic to one of the following graphs:
    \begin{enumerate}
        \item[$(i)$] The infinite path,
        \item[$(ii)$] the cycle graph $C_{n}$ for $n \geq 6$,
        \item[$(iii)$] the dodecahedral graph,
        \item[$(iv)$] the Petersen graph,
        \item[$(v)$] the half-dodecahedral graph,
        \item[$(vi)$] the Triplex graph.
    \end{enumerate}
\end{theorem}

On the other hand, Bhattacharya et al. \cite{Bhattacharya2015} classified all graphs that are $0$-Ricci-flat and have girth at least five.

\begin{theorem}[\cite{Bhattacharya2015}, Corollary 4.1]\thlabel{0_ricci_flat}
    Let $G=(V,E)$ be a locally finite graph with girth at least five. Suppose that $G$ is 0-Ricci-flat. Then $G$ is isomorphic to one of the following graphs:
    \begin{enumerate}
        \item[$(i)$] The infinite path,
        \item[$(ii)$] the infinite ray, 
        \item[$(iii)$] the path $P_{n}$ for $n \geq 2$,
        \item[$(iv)$] the cycle graph $C_{n}$ for $n \geq 5$,
        \item[$(v)$] the star graph $T_{n}$ for $n\geq 3$.   
    \end{enumerate}
\end{theorem}

Combining \thref{0_ricci_flat} and \thref{ricci_flat} yields the following result:

\begin{corollary}\thlabel{bone_idle_girth_5}
    Let $G=(V,E)$ be a locally finite graph with girth at least five. Suppose that $G$ is bone-idle. Then $G$ is isomorphic to one of the following graphs:
    \begin{enumerate}
        \item[$(i)$] The infinite path,
        \item[$(ii)$] the cycle graph $C_{n}$ for $n \geq 6$. 
    \end{enumerate}
\end{corollary}

\begin{remark}
    Hence, there are no bone-idle graphs with girth equal to five.
\end{remark}

The full classification of Ricci-flat and bone-idle graphs appears to be a difficult graph theory problem, which is still open. 
In the following, we leverage our previous findings to investigate the local structure of regular bone-idle graphs.

\subsection{Local structures}

For the subsequent discussion, we associate the following two quantities with an assignment $\phi \in \mathcal{A}_{xy}$:
\begin{itemize}
    \item $N_{1}(\phi)$: The number of neighbors of $x$, forming a 4-cycle based at the edge $x \sim y$, with their image under $\phi$. That is, $N_{1}(\phi) = \vert \{ z \in S_{1}(x)\setminus B_{1}(y): d(z, \phi(z)) = 1\}\vert.$ 
    \item$N_{2}(\phi)$: The number of neighbors of $x$, forming a 5-cycle based at the edge $x \sim y$, with their image under $\phi$. That is,
    $N_{2}(\phi) =\vert \{ z \in S_{1}(x)\setminus B_{1}(y): d(z, \phi(z)) = 2\}\vert .$ 
\end{itemize}

Using this notation, we can examine the local structure of regular Ricci-flat graphs of girth four.

\begin{theorem}\thlabel{ricci_flat_condition}
    Let $G=(V,E)$ be a locally finite graph. Let $x,y\in V$ be of equal degree $d$ with $x\sim y$. Furthermore, assume that $N_{xy} = \emptyset$. Then $\kappa(x,y) = 0$ if and only if one of the following holds:
    \begin{enumerate}
        \item[$(i)$] There exists an optimal assignment $\phi \in \mathcal{O}_{xy}$ such that $N_{1}(\phi) = d-2$ and $N_{2}(\phi) = 0$.
        \item[$(ii)$] There exists an optimal assignment $\phi \in \mathcal{O}_{xy}$ such that $N_{1}(\phi)= d-3$ and $N_{2}(\phi)= 2$.
    \end{enumerate}
\end{theorem}

\begin{proof}
    "$\implies$"Assume $\kappa(x,y) = 0$. Let $\phi \in \mathcal{O}_{xy}$ be an optimal assignment between $S_{1}(x) \setminus B_{1}(y)$ and $S_{1}(y) \setminus B_{1}(x)$. According to \thref{Lin_Lu_Yau_curvature}, we have
    \begin{align*}
        \kappa(x,y) &= \frac{1}{d}\Biggl(d+1 -\mathlarger{\sum}_{z \in S_{1}(x) \setminus B_{1}(y)}d(z, \phi(z)) \Biggr) \\
        &= \frac{1}{d}\Biggl(-2d + 4 + 2 N_{1}(\phi) + N_{2}(\phi)\Biggr) \\
        &= 0.
    \end{align*}
    This can only be the case if one of the cases stated in \thref{ricci_flat_condition} holds true.

    "$\impliedby$" This is an immediate consequence of \thref{Lin_Lu_Yau_curvature}.
\end{proof}

Next, we present a necessary and sufficient condition for an edge $x\sim y$ to have $\kappa_{0}(x,y)=0$. This condition was already established by Bhattacharya et al. \cite{Bhattacharya2015}.

\begin{theorem}
    Let $G=(V,E)$ be a locally finite graph. Let $x,y\in V$ be of equal degree $d$ with $x\sim y$. Furthermore, assume that $N_{xy} = \emptyset$. Then $\kappa_{0}(x,y) = 0$ if and only if there exists a perfect matching between $S_{1}(x)$ and $S_{1}(y)$.
\end{theorem}

Therefore, a $0$-Ricci-flat, regular graph of girth four must have a perfect matching between the neighborhoods $S_{1}(x)$ and $S_{1}(y)$ for every edge $x\sim y$. Examples are the $n$-dimensional hypercube $\mathcal{Q}_{n}$, the $n$-dimensional integer lattice $\mathbb{Z}^{n}$ and the complete bipartite graph $K_{n,n}$.

The subsequent theorem provides a necessary and sufficient condition for an edge in a graph of girth four to be bone-idle.

\begin{theorem}\thlabel{bone_idle}
    Let $G=(V,E)$ be a locally finite graph. Let $x,y\in V$ be of equal degree $d$ with $x\sim y$. Furthermore, assume that $N_{xy} = \emptyset$. Then the edge $x\sim y$ is bone-idle if and only if there exists an optimal assignment $\phi \in \mathcal{O}_{xy}$ such that $N_{1}(\phi) = d-2$ and $N_{2}(\phi) = 0$.
\end{theorem}

\begin{proof}
    Assume the edge $x \sim y$ is bone-idle, i.e., $\kappa(x,y) = \kappa_{0}(x,y) = 0$. 
    According to \thref{equality_condition} there exists an optimal assignment $\phi \in \mathcal{O}_{xy}$ between $S_{1}(x)\setminus B_{1}(y)$ and $S_{1}(y)\setminus B_{1}(x)$ and a $z \in S_{1}(x)\setminus B_{1}(y)$, such that $d(z,\phi(z)) = 3$. Therefore, 
    \begin{equation}\label{eq:6}
        N_{1}(\phi) + N_{2}(\phi) < \vert S_{1}(x)\setminus B_{1}(y) \vert = d-1.
    \end{equation}
    As $\kappa(x,y) = 0$ we can apply \thref{ricci_flat_condition}. Using equation \ref{eq:6}, we conclude that $N_{1}(\phi) = d-2$ and $N_{2}(\phi) = 0$ must hold.
    
    Conversely, assume that there exists an optimal assignment $\phi \in \mathcal{O}_{xy}$ that satisfies $N_{1}(\phi) = d-2$ and $N_{2}(\phi) = 0$. Since 
    \begin{equation*}
        N_{1}(\phi) + N_{2}(\phi) = d-2 < \vert S_{1}(x)\setminus B_{1}(y) \vert,
    \end{equation*} 
    there exists a $z \in S_{1}(x)\setminus B_{1}(y)$ such that $d(z, \phi(z)) = 3$. Hence, according to \thref{equality_condition}, we have $\kappa(x,y) = \kappa_{0}(x,y)$.
    According to \thref{ricci_flat_condition}, we have $\kappa(x,y) = 0$, which concludes the proof.
\end{proof}

Both the complete bipartite graph $K_{n,n}$ and the $n$-dimensional hypercube $\mathcal{Q}_{n}$ are $0$-Ricci-flat, regular graphs of girth four. Note that there exists a perfect matching between $S_{1}(x)\setminus B_{1}(y)$ and $S_{1}(y)\setminus B_{1}(x)$ for any edge $x \sim y$ in the complete bipartite graph $K_{n,n}$. The same holds for the $n$-dimensional hypercube $\mathcal{Q}_{n}$. Therefore, according to the previous theorem, neither of the graphs is bone-idle. Using \thref{Lin_Lu_Yau_curvature} for the Lin-Lu-Yau curvature, we obtain that for both the complete bipartite graph $K_{n,n}$ and the $n$-dimensional hypercube $\mathcal{Q}_{n}$
\begin{equation*}
    \kappa(x,y) = \frac{2}{n},
\end{equation*}
for every edge $x \sim y$. Thus, the graphs $K_{n,n}$ and $\mathcal{Q}_{n}$ satisfy $\kappa(x,y) >0$ and $\kappa_{0}(x,y)=0$ for all edges $x \sim y$.

A direct implication of \thref{bone_idle} is that the $n$-dimensional integer lattice $\mathbb{Z}^{n}$ is bone-idle.
 
We conclude this section by extending \thref{bone_idle} to arbitrary regular graphs as follows.

\begin{theorem}
    Let $G=(V,E)$ be a locally finite graph. Let $x,y \in V$ be of equal degree $d$ with $x \sim y$. Then the edge $x \sim y$ is bone-idle if and only if there exists an optimal assignment $\phi \in \mathcal{O}_{xy}$ such that $N_{1}(\phi) + N_{2}(\phi) < d - 1 - \vert N_{xy} \vert$ and 
    \begin{equation*}
        2d - 4 - 3 \vert N_{xy}\vert =  2 N_{1}(\phi) + N_{2}(\phi).
    \end{equation*} 
\end{theorem}

\begin{proof}
    "$\implies$" Assume the edge $x \sim y$ is bone-idle, i.e., $\kappa(x,y) = \kappa_{0}(x,y) = 0$. According to \thref{equality_condition} there exists an optimal assignment $\phi \in \mathcal{O}_{xy}$ between $S_{1}(x)\setminus B_{1}(y)$ and $S_{1}(y)\setminus B_{1}(x)$ and a $z \in S_{1}(x)\setminus B_{1}(y)$, such that $d(z,\phi(z)) = 3$. Therefore, 
    \begin{equation}
        N_{1}(\phi) + N_{2}(\phi) < \vert S_{1}(x)\setminus B_{1}(y)\vert = d - 1 - \vert N_{xy} \vert .
    \end{equation}
    Using $\kappa(x,y) = 0$ and \thref{Lin_Lu_Yau_curvature}, we obtain
    \begin{align*}
        \kappa(x,y) &= \frac{1}{d}\Biggl(d+1 -\mathlarger{\sum}_{z \in S_{1}(x) \setminus B_{1}(y)}d(z, \phi(z)) \Biggr) \\
        &= \frac{1}{d}\Biggl(-2d + 4 + 3 \vert N_{xy} \vert + 2 N_{1}(\phi) + N_{2}(\phi)\Biggr) \\
        &= 0,
    \end{align*}
    or equivalently, $2N_{1}(\phi) + N_{2}(\phi) = 2d - 4 - 3 \vert N_{xy} \vert$.

    "$\impliedby$" Assume $\phi$ is an optimal assignment such that 
    \begin{equation*}
        N_{1}(\phi) + N_{2}(\phi) < d - 1 - \vert N_{xy} \vert = \vert S_{1}(x) \setminus B_{1}(y) \vert.
    \end{equation*} 
    Therefore, there exists an $z\in S_{1}(x) \setminus B_{1}(y)$ such that $d(z, \phi(z)) = 3$. Hence, according to \thref{equality_condition}, we have $\kappa(x,y) = \kappa_{0}(x,y)$. Using \thref{Lin_Lu_Yau_curvature} and 
    \begin{equation*}
        2d - 4 - 3 \vert N_{xy}\vert =  + 2 N_{1}(\phi) + N_{2}(\phi),
    \end{equation*}
    we obtain $\kappa(x,y) = 0$, and the edge $x\sim y$ is bone-idle.
\end{proof}

\subsection{3-regular bone-idle graphs}

The objective of this section is to demonstrate that no 3-regular bone-idle graphs exist.

\begin{theoremBoneIdle}
    Let $G=(V,E)$ be a locally finite graph. Suppose that $G$ is bone-idle, then $G$ is not $3$-regular.
\end{theoremBoneIdle}

\begin{figure}
    \center 
    \begin{subfigure}{0.4\textwidth}
        \centering
        \begin{tikzpicture}[x=1.5cm, y=1.5cm,
            vertex/.style={
                shape=circle, fill=black, inner sep=1.5pt	
            }
        ]
        
        \node[vertex, label=below:$x$] (1) at (0, 0) {};
        \node[vertex, label=below:$y$] (2) at (1, 0) {};
        \node[vertex, label=above:$x_{1}$] (3) at (0, 1) {};
        \node[vertex, label=above: $y_{1}$] (4) at (1, 1) {};
        \node[vertex, label=below:$x_{2}$] (5) at (-1, 0) {};
        \node[vertex, label=above:$z$] (6) at (-1, 1) {};
        \node[vertex, label=above:$y_{2}$] (7) at (2, 0) {};
        
        \draw (1) -- (2);
        \draw (1) -- (5);
        \draw[red] (1) -- node[midway,right] {$\kappa>0$} ++ (3);
        \draw (2) -- (4);
        \draw (3) -- (4);
        \draw (5) -- (6);
        \draw (3) -- (6);   
        \draw (2) -- (7);    
        \end{tikzpicture}
      \end{subfigure}
      \begin{subfigure}{0.4\textwidth}
        \centering
        \begin{tikzpicture}[x=1.5cm, y=1.5cm,
            vertex/.style={
                shape=circle, fill=black, inner sep=1.5pt	
            }
        ]
        
        \node[vertex, label=left:$x$] (1) at (0, 0) {};
        \node[vertex, label=right:$y$] (2) at (1, 0) {};
        \node[vertex, label=above:$x_{1}$] (3) at (0, 1) {};
        \node[vertex, label=above:$y_{1}$] (4) at (1, 1) {};
        \node[vertex, label=below:$x_{2}$] (5) at (0, -1) {};
        \node[vertex, label=below:$y_{2}$] (6) at (1, -1) {};
        
        \draw[red] (1) -- node[midway,below] {$\kappa>0$} ++ (2);
        \draw (1) -- (3);
        \draw (1) -- (5);
        \draw (2) -- (4);
        \draw (2) -- (6);
        \draw (3) -- (4);
        \draw (5) -- (6);

        \end{tikzpicture}
      \end{subfigure}
    \caption{Illustration of the two possible cases in \thref{bone_idle_3_reg}}
\end{figure}
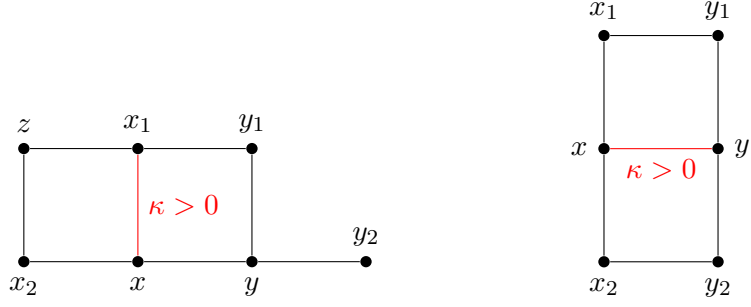

\begin{proof}
    We argue by contradiction. Assume $G$ is $3$-regular and bone-idle. According to \thref{bone_idle_girth_5}, the girth of $G$ must be less than five. Assume there exists an edge $x \sim y$ such that $\vert N_{xy} \vert > 0$. Then $S_{1}(x)\setminus B_{1}(y)$ and 
    $S_{1}(y)\setminus B_{1}(x)$ each contain only a single vertex, which we denote by $z_{1}$ and $z_{2}$, respectively. Recall that $d(z_{1}, z_{2}) \leq 3$. Thus,
    \begin{equation*}
        \kappa(x,y) = \frac{1}{3}\Big(4 - d(z_{1},z_{2})\Big) > 0.
    \end{equation*}
    This contradicts the bone-idleness of the graph. Hence, the girth of $G$ must be equal to four. Let $x\sim y$ be an edge contained in a $4$-cycle. Denote by $x_{1}, x_{2}$ and $y_{1}, y_{2}$ the other two neighbors of $x$ and $y$, respectively. Without loss of generality, we may assume that $x_{1} \sim y_{1}$, as $x \sim y$ is contained in a $4$-cycle. Observe that $x_{2} \sim y_{1}$ and $y_{2} \sim x_{1}$ cannot hold true at the same time. Otherwise, there exists a perfect matching between $S_{1}(x)\setminus\{y\}$ and $S_{1}(y)\setminus\{x\}$, leading to $\kappa(x,y) = \frac{2}{3} > 0$. 
    Therefore, we may assume, without loss of generality, that $x_{2} \not\sim y_{1}$.

    According to \thref{bone_idle}, there must be a $4$-cycle based on the edge $x \sim x_{2}$. Using that $x_{2} \not\sim y_{1}$, one of the following cases must be true:

    \emph{Case 1:} There is a $z \in S_{1}(x_{2})\setminus\{x\}$ such that $z \sim x_{1}$. In this case we have $\kappa(x, x_{1})>0$, contradicting the bone-idleness of $G$. 

    \emph{Case 2:} $x_{2} \sim y_{2}$. In this case, we have $\kappa(x,y) >0$, contradicting the bone-idleness of $G$. 
    
    This concludes the proof.
\end{proof}

Hence, there are no 3-regular bone-idle graphs.

\subsection{4-regular bone-idle graphs}\label{bone_idle_4_reg}

\begin{figure}
    \center 
    \begin{tikzpicture}[x=1.5cm, y=1.5cm,
        vertex/.style={
            shape=circle, fill=black, inner sep=1.5pt   
        }
    ]
    
        \node[vertex] (1) at (-0.8, 0) {};
        \node[vertex] (2) at (0.8, 0) {};
        \node[vertex] (3) at (-1.3, 1.5) {};
        \node[vertex] (4) at (1.3, 1.5) {};
        \node[vertex] (5) at (0, 2.5) {};

        \node[vertex] (6) at (0, 0.25) {};
        \node[vertex] (7) at (0.8, 0.8) {};
        \node[vertex] (8) at (-0.8, 0.8) {};
        \node[vertex] (9) at (-0.55, 1.8) {};
        \node[vertex] (10) at (0.55, 1.8) {};

        \node[vertex] (11) at (0.2, 1.6) {};
        \node[vertex] (12) at (-0.2, 1.6) {};
        \node[vertex] (13) at (-0.4, 1.35) {};
        \node[vertex] (14) at (0.4, 1.35) {};
        \node[vertex] (15) at (-0.5, 1.1) {};
        \node[vertex] (16) at (0.5, 1.1) {};
        \node[vertex] (17) at (-0.4, 0.85) {};
        \node[vertex] (18) at (0.4, 0.85) {};
        \node[vertex] (19) at (0.2, 0.6) {};
        \node[vertex] (20) at (-0.2, 0.6) {};

        \node[vertex] (21) at (0, 1.4) {};
        \node[vertex] (22) at (-0.3, 1.2) {};
        \node[vertex] (23) at (-0.1, 1.25) {};
        \node[vertex] (24) at (0.1, 1.25) {};
        \node[vertex] (25) at (0.3, 1.2) {};
        \node[vertex] (26) at (-0.18, 1.05) {};
        \node[vertex] (27) at (0.18, 1.05) {};
        \node[vertex] (28) at (0, 1) {};
        \node[vertex] (29) at (0.2, 0.8) {};
        \node[vertex] (30) at (-0.2, 0.8) {};
        
        \draw (1) -- (2);
        \draw (1) -- (3);
        \draw (2) -- (4);
        \draw (3) -- (5);
        \draw (4) -- (5);
        \draw (1) -- (6);
        \draw (2) -- (6);
        \draw (2) -- (7);
        \draw (4) -- (7);
        \draw (1) -- (8);
        \draw (3) -- (8);
        \draw (3) -- (9);
        \draw (5) -- (9);
        \draw (5) -- (10);
        \draw (4) -- (10);
        \draw (11) -- (12);
        \draw (12) -- (13);
        \draw (11) -- (14);
        \draw (14) -- (16);
        \draw (13) -- (15);
        \draw (15) -- (17);
        \draw (16) -- (18);
        \draw (18) -- (19);
        \draw (17) -- (20);
        \draw (19) -- (20);
        \draw (19) -- (6);
        \draw (20) -- (6);
        \draw (12) -- (9);
        \draw (13) -- (9);
        \draw (11) -- (10);
        \draw (14) -- (10);
        \draw (15) -- (8);
        \draw (17) -- (8);
        \draw (18) -- (7);
        \draw (16) -- (7);
        \draw (11) -- (21);
        \draw (12) -- (21);
        \draw (13) -- (22);
        \draw (15) -- (22);
        \draw (22) -- (23);
        \draw (23) -- (24);
        \draw (21) -- (23);
        \draw (21) -- (24);
        \draw (24) -- (25);
        \draw (25) -- (16);
        \draw (25) -- (14);
        \draw (22) -- (26);
        \draw (23) -- (26);
        \draw (24) -- (27);
        \draw (25) -- (27);
        \draw (27) -- (28);
        \draw (26) -- (28);
        \draw (28) -- (29);
        \draw (27) -- (29);
        \draw (29) -- (19);
        \draw (30) -- (20);
        \draw (30) -- (28);
        \draw (30) -- (26);
        \draw (29) -- (18);
        \draw (30) -- (17);
        
    \end{tikzpicture}
    \caption{Illustration of the icosidodecahedron graph}
    \label{icosidodecahedron_graph}
\end{figure}
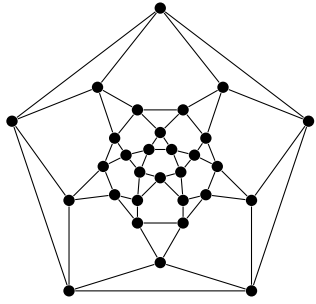

In this section, we aim to provide a complete classification of all 4-regular bone-idle graphs. We begin by examining graphs with a girth of three. In \cite{Bai2021}, the authors classify all Ricci-flat 4-regular graphs of girth three.

\begin{theorem}[\cite{Bai2021}, Theorem 5]
    Let $G=(V,E)$ be a 4-regular graph of girth three. If $G$ is Ricci-flat, then it is isomorphic to the icosidodecahedron graph.
\end{theorem}

\begin{remark}
    See Figure \ref{icosidodecahedron_graph} for an illustration of the icosidodecahedron graph, a polyhedron with 20 triangular faces, 12 pentagonal faces on 30 vertices connected by 60 identical edges, each of which separates a triangle from a pentagon.
\end{remark}

\begin{corollary}
    Let $G=(V,E)$ be a 4-regular graph of girth three. If $G$ is bone-idle, then it is isomorphic to the icosidodecahedron graph.
\end{corollary}

\begin{proof}
    Assume $G$ is a 4-regular bone-idle graph of girth three. Since it is bone-idle, $G$ is Ricci-flat and therefore must be isomorphic to the icosidodecahedron graph. It remains to verify that the the icosidodecahedron graph $G$ is indeed bone-idle. To this end, let $x\sim y$ be an arbitrary edge in $G$. Let $\phi \in \mathcal{O}_{xy}$ be an optimal assignment between $S_{1}(x)\setminus B_{1}(y)$ and $S_{1}(y)\setminus B_{1}(x)$. Then there exists an $z \in S_{1}(x)\setminus B_{1}(y)$ such that $d(z,\phi(z)) = 3$. According to \thref{equality_condition}, we have $\kappa(x,y) = \kappa_{0}(x,y)$. Using the Ricci-flatness we have $\kappa(x,y)=0$. Using \thref{bone-idle_vs_ricci-flat}, we conclude that the edge $x\sim y$ is bone-idle. 
\end{proof}

Therefore, we have classified all 4-regular bone-idle graphs of girth three. We now proceed to classify all such graphs of girth four. 

\begin{figure}
    \centering
   
    \begin{tikzpicture}[x=1.5cm, y=1.5cm,
            vertex/.style={
                shape=circle, fill=black, inner sep=1.5pt	
            }
        ]
        \node[vertex] (1) at (0, 0) {};
        \node[vertex] (2) at (1, 0) {};
        \node[vertex] (3) at (2, 0) {};
        \node[vertex] (4) at (3, 0) {};
        \node[vertex] (5) at (4, 0) {};
        \node[vertex] (6) at (5, 0) {};
        \node[vertex] (7) at (6, 0) {};

        \node[vertex] (8) at (0, 1) {};
        \node[vertex] (9) at (1, 1) {};
        \node[vertex] (10) at (2, 1) {};
        \node[vertex] (11) at (3, 1) {};
        \node[vertex] (12) at (4, 1) {};
        \node[vertex] (13) at (5, 1) {};
        \node[vertex] (14) at (6, 1) {};

        \draw (1) -- (2);
        \draw (2) -- (3);
        \draw (3) -- (4);
        \draw (4) -- (5);
        \draw (5) -- (6);
        \draw (6) -- (7);

        \draw (8) -- (9);
        \draw (9) -- (10);
        \draw (10) -- (11);
        \draw (11) -- (12);
        \draw (12) -- (13);
        \draw (13) -- (14);

        \draw (1) -- (9);
        \draw (2) -- (8);
        \draw (2) -- (10);
        \draw (3) -- (9);
        \draw (3) -- (11);
        \draw (4) -- (10);
        \draw (4) -- (12);
        \draw (5) -- (11);
        \draw (5) -- (13);
        \draw (6) -- (12);
        \draw (6) -- (14);
        \draw (7) -- (13);

        \draw[dashed] (1) -- (-1,0);
        \draw[dashed] (1) -- (-1,1);
        \draw[dashed] (8) -- (-1,1);
        \draw[dashed] (8) -- (-1,0);

        \draw[dashed] (7) -- (7,0);
        \draw[dashed] (7) -- (7,1);
        \draw[dashed] (14) -- (7,1);
        \draw[dashed] (14) -- (7,0);
    \end{tikzpicture}
    \caption{A primitive 4-regular Ricci-flat graph}
    \label{primitive_one}
\end{figure}
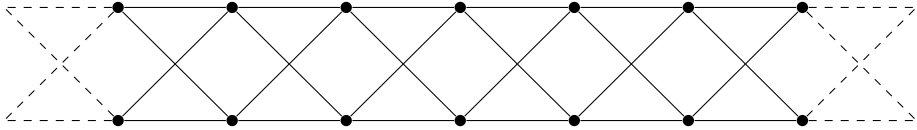

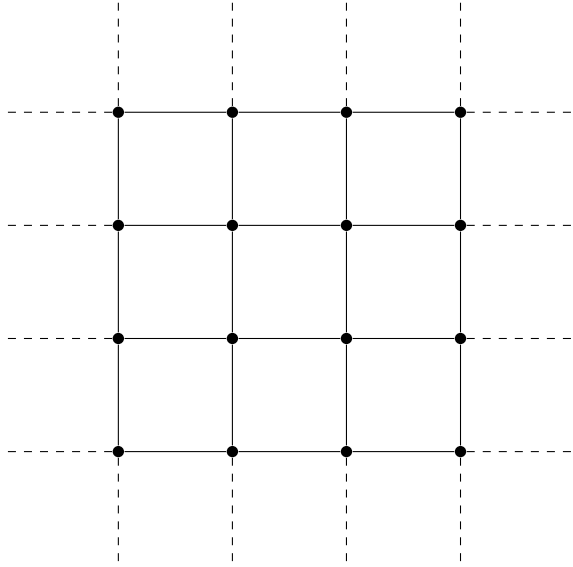
\begin{figure}
    \centering
   
    \begin{tikzpicture}[x=1.5cm, y=1.5cm,
            vertex/.style={
                shape=circle, fill=black, inner sep=1.5pt	
            }
        ]
        \node[vertex] (1) at (0, 0) {};
        \node[vertex] (2) at (1, 0) {};
        \node[vertex] (3) at (2, 0) {};
        \node[vertex] (4) at (3, 0) {};

        \node[vertex] (6) at (0, 1) {};
        \node[vertex] (7) at (1, 1) {};
        \node[vertex] (8) at (2, 1) {};
        \node[vertex] (9) at (3, 1) {};
        
        \node[vertex] (11) at (0, 2) {};
        \node[vertex] (12) at (1, 2) {};
        \node[vertex] (13) at (2, 2) {};
        \node[vertex] (14) at (3, 2) {};

        \node[vertex] (16) at (0, 3) {};
        \node[vertex] (17) at (1, 3) {};
        \node[vertex] (18) at (2, 3) {};
        \node[vertex] (19) at (3, 3) {};


        \draw (1) -- (2);
        \draw (2) -- (3);
        \draw (3) -- (4);

        \draw (6) -- (7);
        \draw (7) -- (8);
        \draw (8) -- (9);

        \draw (1) -- (6);
        \draw (2) -- (7);
        \draw (3) -- (8);
        \draw (4) -- (9);

        \draw (11) -- (12);
        \draw (12) -- (13);
        \draw (13) -- (14);

        \draw (11) -- (6);
        \draw (12) -- (7);
        \draw (13) -- (8);
        \draw (14) -- (9);

        \draw (16) -- (17);
        \draw (17) -- (18);
        \draw (18) -- (19);

        \draw (11) -- (16);
        \draw (12) -- (17);
        \draw (13) -- (18);
        \draw (14) -- (19);



        \draw[dashed] (1) -- (-1,0);
        \draw[dashed] (4) -- (4,0);
        \draw[dashed] (6) -- (-1,1);
        \draw[dashed] (9) -- (4,1);
        \draw[dashed] (11) -- (-1,2);
        \draw[dashed] (14) -- (4,2);
        \draw[dashed] (16) -- (-1,3);
        \draw[dashed] (19) -- (4,3);

        \draw[dashed] (1) -- (0,-1);
        \draw[dashed] (2) -- (1,-1);
        \draw[dashed] (3) -- (2,-1);
        \draw[dashed] (4) -- (3,-1);

        \draw[dashed] (16) -- (0,4);
        \draw[dashed] (17) -- (1,4);
        \draw[dashed] (18) -- (2,4);
        \draw[dashed] (19) -- (3,4);

    \end{tikzpicture}
    \caption{A primitive 4-regular Ricci-flat graph of "lattice type"}
    \label{primitive_two}
\end{figure}

In \cite{Bai2021}, the authors classify all 4-regular Ricci-flat graphs that contain two four-cycles sharing a common edge. To this end, they introduce the concept of a \textit{primitive graph}, which can be understood as the 1-skeleton of the universal cover of the CW-complex formed by gluing 2-cells to all cycles of length at most five. They obtain the following result.

\begin{theorem}[\cite{Bai2021}, Theorem 8]
    Let $G=(V,E)$ be a $4$-regular Ricci-flat  graph that contains two four-cycles sharing one edge. Then G is isomorphic to graphs with the primitive graphs showing in Figure \ref{primitive_one} and Figure \ref{primitive_two}.
\end{theorem}

According to \thref{bone_idle}, there must be two four-cycles supported on every edge in a 4-regular bone-idle graph of girth four. Thus, we can apply the previous theorem. Using \thref{bone_idle}, it is easy to verify that the graphs with the primitive graphs showing in Figure \ref{primitive_one} and Figure \ref{primitive_two} are bone-idle. Therefore, we obtain the following Corollary.

\begin{corollary}
    Let $G=(V,E)$ be a 4-regular bone-idle graph of girth four. Then $G$ is isomorphic to graphs with the primitive graphs showing in Figure \ref{primitive_one} and Figure \ref{primitive_two}.
\end{corollary}

The finite graphs with the primitive graph depicted in Figure \ref{primitive_one} can be constructed as follows: Start with an $n$-cycle consisting of vertices $x_{0}, \dots, x_{n-1}$ placed inside another $n$-cycle with vertices $y_{0}, \dots, y_{n-1}$. Connect each vertex $y_{k}$ to $x_{(k-1) \mod n}$ and $x_{(k+1) \mod n}$. These graphs are denoted by $BI_{n}$. Figure \ref{bone_idle_graphs} illustrates the graphs $BI_{6}$, $BI_{7}$, and $BI_{8}$.

\begin{figure}
    \center 
    \begin{subfigure}{0.4\textwidth}
        \centering
        \begin{tikzpicture}[x=1.5cm, y=1.5cm,
            vertex/.style={
                shape=circle, fill=black, inner sep=1.5pt	
            }
        ]
        
        \node[vertex] (1) at (0, 0) {};
        \node[vertex] (7) at (-0.5, 0) {};
        \node[vertex] (2) at (0.5, 0.5) {};
        \node[vertex] (8) at (0.5, 1) {};
        \node[vertex] (3) at (1, 0.5) {};
        \node[vertex] (9) at (1, 1) {};
        \node[vertex] (4) at (1.5, 0) {};
        \node[vertex] (10) at (2, 0) {};
        \node[vertex] (5) at (1, -0.5) {};
        \node[vertex] (11) at (1, -1) {};
        \node[vertex] (6) at (0.5, -0.5) {};
        \node[vertex] (12) at (0.5, -1) {};

        \draw (1) -- (2);
        \draw (2) -- (3);
        \draw (3) -- (4);
        \draw (4) -- (5);
        \draw (5) -- (6);
        \draw (6) -- (1);
        \draw (7) -- (8);
        \draw (8) -- (9);
        \draw (9) -- (10);
        \draw (10) -- (11);
        \draw (11) -- (12);
        \draw (12) -- (7);
        \draw (7) -- (2);
        \draw (7) -- (6);
        \draw (8) -- (1);
        \draw (8) -- (3);
        \draw (9) -- (2);
        \draw (9) -- (4);
        \draw (10) -- (3);
        \draw (10) -- (5);
        \draw (11) -- (4);
        \draw (11) -- (6);
        \draw (12) -- (5);
        \draw (12) -- (1);

        \end{tikzpicture}
      \end{subfigure}
      \begin{subfigure}{0.4\textwidth}
        \centering
        \begin{tikzpicture}[x=1.5cm, y=1.5cm,
            vertex/.style={
                shape=circle, fill=black, inner sep=1.5pt	
            }
        ]
        \node[vertex] (1) at (0, 0.25) {};
        \node[vertex] (2) at (0.5, 0.5) {};
        \node[vertex] (3) at (1, 0.5) {};
        \node[vertex] (4) at (1.5, 0) {};
        \node[vertex] (5) at (1, -0.5) {};
        \node[vertex] (6) at (0.5, -0.5) {};
        \node[vertex] (7) at (0, -0.25) {};
        \node[vertex] (8) at (-0.5, 0.25) {};
        \node[vertex] (9) at (0.5, 1) {};
        \node[vertex] (10) at (1, 1) {};
        \node[vertex] (11) at (2, 0) {};
        \node[vertex] (12) at (1, -1) {};
        \node[vertex] (13) at (0.5, -1) {};
        \node[vertex] (14) at (-0.5, -0.25) {};

        \draw (1) -- (2);
        \draw (2) -- (3);
        \draw (3) -- (4);
        \draw (4) -- (5);
        \draw (5) -- (6);
        \draw (6) -- (7);
        \draw (7) -- (1);

        \draw (8) -- (9);
        \draw (9) -- (10);
        \draw (10) -- (11);
        \draw (11) -- (12);
        \draw (12) -- (13);
        \draw (13) -- (14);
        \draw (14) -- (8);

        \draw (8) -- (7);
        \draw (8) -- (2);
        \draw (9) -- (1);
        \draw (9) -- (3);
        \draw (10) -- (2);
        \draw (10) -- (4);
        \draw (11) -- (3);
        \draw (11) -- (5);
        \draw (12) -- (4);
        \draw (12) -- (6);
        \draw (13) -- (5);
        \draw (13) -- (7);
        \draw (14) -- (1);
        \draw (14) -- (6);

    \end{tikzpicture}
    \end{subfigure}
      \begin{subfigure}{0.4\textwidth}
        \centering
        \begin{tikzpicture}[x=1.5cm, y=1.5cm,
            vertex/.style={
                shape=circle, fill=black, inner sep=1.5pt	
            }
        ]
        \node[vertex] (1) at (0, 0.25) {};
        \node[vertex] (9) at (-0.5, 0.25) {};
        \node[vertex] (8) at (0, -0.25) {};
        \node[vertex] (16) at (-0.5, -0.25) {};
        \node[vertex] (2) at (0.5, 0.5) {};
        \node[vertex] (10) at (0.5, 1) {};
        \node[vertex] (3) at (1, 0.5) {};
        \node[vertex] (11) at (1, 1) {};
        \node[vertex] (4) at (1.5, 0.25) {};
        \node[vertex] (5) at (1.5, -0.25) {};
        \node[vertex] (12) at (2, 0.25) {};
        \node[vertex] (13) at (2, -0.25) {};
        \node[vertex] (6) at (1, -0.5) {};
        \node[vertex] (14) at (1, -1) {};
        \node[vertex] (7) at (0.5, -0.5) {};
        \node[vertex] (15) at (0.5, -1) {};

        \draw (1) -- (2);
        \draw (2) -- (3);
        \draw (3) -- (4);
        \draw (4) -- (5);
        \draw (5) -- (6);
        \draw (6) -- (7);
        \draw (7) -- (8);
        \draw (8) -- (1);

        \draw (9) -- (10);
        \draw (10) -- (11);
        \draw (11) -- (12);
        \draw (12) -- (13);
        \draw (13) -- (14);
        \draw (14) -- (15);
        \draw (15) -- (16);
        \draw (16) -- (9);

        \draw (9) -- (2);
        \draw (9) -- (8);
        \draw (10) -- (1);
        \draw (10) -- (3);
        \draw (11) -- (2);
        \draw (11) -- (4);
        \draw (12) -- (3);
        \draw (12) -- (5);
        \draw (13) -- (4);
        \draw (13) -- (6);
        \draw (14) -- (5);
        \draw (14) -- (7);
        \draw (15) -- (6);
        \draw (15) -- (8);
        \draw (16) -- (7);
        \draw (16) -- (1);

    \end{tikzpicture}
    \end{subfigure}
    \caption{Illustration of the graphs $BI_{6}, BI_{7}$ and $BI_{8}$}
    \label{bone_idle_graphs}
\end{figure}
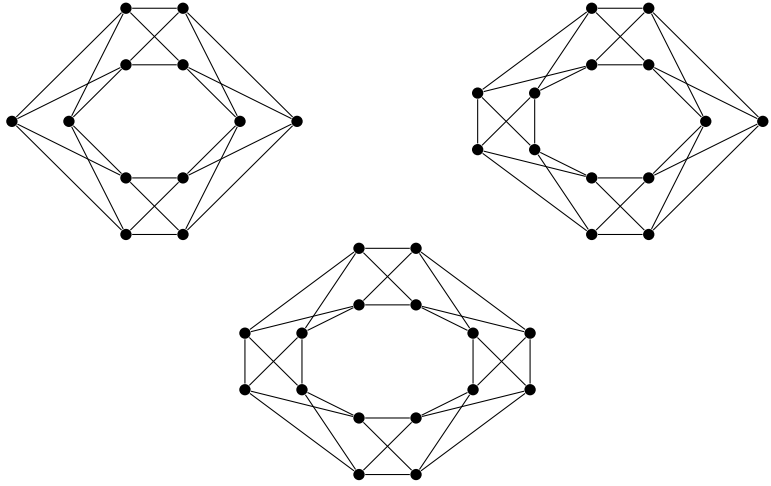

Examples of finite graphs with the primitive graph shown in Figure \ref{primitive_two} include potentially twisted tori and the Klein bottle graphs. These graphs can be constructed as follows: Take an $n\times m$ grid graph and denote the vertices by $x_{i,j}$ for $i=0,\dots,n-1$ and $j=0,\dots, m-1$. First, add the edges $x_{0,j} \sim x_{n-1,j}$ for $j=0,\dots, m-1$. Then, to obtain a Klein bottle graph add the edges $x_{i,0}\sim x_{n-1-i, m-1}$ for $i=0,\dots n-1$. Note that $n,m \geq 6$ must hold. On the other hand, to obtain a twisted torus graph add the edges $x_{i,0} \sim x_{(i+l)\mod n,m-1}$ for $i=0,\dots,n-1$ and for a fixed $l \leq \frac{n}{2}$. Note that $n\geq 6$ and $m+l \geq 6$ must hold. A twisted torus graph and a Klein bottle graph are depicted in Figure \ref{twisted_torus_klein_bottle}.

By \thref{bone_idle_girth_5}, no 4-regular bone-idle graphs exist with girth greater than or equal to five. Consequently, we have completed the classification of all 4-regular bone-idle graphs.

\subsection{5-regular bone-idle graphs}

In \cite{LLY2011}, Lin, Lu, and Yau established the following result concerning the curvature of Cartesian products of graphs.

\begin{theorem}[\cite{LLY2011}]
    Let $G=(V_{G},E_{G})$ be a $d_{G}$-regular graph and $H=(V_{H}, E_{H})$ be a $d_{H}$-regular graph. Let $x_{1}, x_{2}\in V_{G}$ with $x_{1}\sim x_{2}$ and $y \in V_{H}$. Then 
    \begin{align*}
        \kappa^{G\square H}_{0}((x_{1},y),(x_{2},y)) & = \frac{d_{G}}{d_{G}+d_{H}} \kappa^{G}_{0}(x_{1},x_{2}),
        \\
        \kappa^{G\square H}((x_{1},y),(x_{2},y)) & = \frac{d_{G}}{d_{G}+d_{H}} \kappa^{G}(x_{1},x_{2}).  
    \end{align*}
\end{theorem}

\begin{corollary}
    Let $G=(V_{G},E_{G})$ be a $d_{G}$-regular graph and $H=(V_{H}, E_{H})$ be a $d_{H}$-regular graph. If $G$ and $H$ are Ricci-flat graphs, then the Cartesian product $G \square H$ is also Ricci-flat.
\end{corollary}

\begin{figure}
    \center 
    \begin{subfigure}{0.4\textwidth}
        \centering
        \begin{tikzpicture}[x=1.5cm, y=1.5cm,
            vertex/.style={
                shape=circle, fill=black, inner sep=1.5pt	
            }
        ]
        
        \node[vertex] (1) at (0, 0) {};
        \node[vertex] (2) at (0.5, 0) {};
        \node[vertex] (3) at (1, 0) {};

        \node[vertex] (4) at (0, 0.5) {};
        \node[vertex] (5) at (0.5, 0.5) {};
        \node[vertex] (6) at (1, 0.5) {};
        
        \node[vertex] (7) at (0, 1) {};
        \node[vertex] (8) at (0.5, 1) {};
        \node[vertex] (9) at (1, 1) {};

        \node[vertex] (10) at (0, 1.5) {};
        \node[vertex] (11) at (0.5, 1.5) {};
        \node[vertex] (12) at (1, 1.5) {};

        \node[vertex] (13) at (0, 2) {};
        \node[vertex] (14) at (0.5, 2) {};
        \node[vertex] (15) at (1, 2) {};

        \node[vertex] (16) at (0, 2.5) {};
        \node[vertex] (17) at (0.5, 2.5) {};
        \node[vertex] (18) at (1, 2.5) {};

        \draw (1) -- (2);
        \draw (2) -- (3);

        \draw (4) -- (5);
        \draw (5) -- (6);

        \draw (7) -- (8);
        \draw (8) -- (9);

        \draw (10) -- (11);
        \draw (11) -- (12);

        \draw (13) -- (14);
        \draw (14) -- (15);

        \draw (16) -- (17);
        \draw (17) -- (18);
        
        \draw (1) -- (4);
        \draw (4) -- (7);
        \draw (7) -- (10);
        \draw (10) -- (13);
        \draw (13) -- (16);

        \draw (2) -- (5);
        \draw (5) -- (8);
        \draw (8) -- (11);
        \draw (11) -- (14);
        \draw (14) -- (17);

        \draw (3) -- (6);
        \draw (6) -- (9);
        \draw (9) -- (12);
        \draw (12) -- (15);
        \draw (15) -- (18);

        \draw    (1) .. controls (-0.25,1.25) and (-0.25,1.25) .. (16) ;
        \draw    (2) .. controls (0.25,1.25) and (0.25,1.25) .. (17) ;
        \draw    (3) .. controls (0.75,1.25) and (0.75,1.25) .. (18) ;

        \draw    (1) -- (12) ;
        \draw    (4) -- (15) ;
        \draw    (7) -- (18) ;
        \draw    (10) -- (3) ;
        \draw    (13) -- (6) ;
        \draw    (16) -- (9) ;
      
        \end{tikzpicture}
    \end{subfigure}
    \begin{subfigure}{0.4\textwidth}
        \centering
        \begin{tikzpicture}[x=1.5cm, y=1.5cm,
            vertex/.style={
                shape=circle, fill=black, inner sep=1.5pt	
            }
        ]
        
        \node[vertex] (1) at (0, 0) {};
        \node[vertex] (2) at (0.5, 0) {};
        \node[vertex] (3) at (1, 0) {};
        \node[vertex] (4) at (1.5, 0) {};
        \node[vertex] (5) at (2, 0) {};
        \node[vertex] (6) at (2.5, 0) {};

        \node[vertex] (7) at (0, 0.5) {};
        \node[vertex] (8) at (0.5, 0.5) {};
        \node[vertex] (9) at (1, 0.5) {};
        \node[vertex] (10) at (1.5, 0.5) {};
        \node[vertex] (11) at (2, 0.5) {};
        \node[vertex] (12) at (2.5, 0.5) {};

        \node[vertex] (13) at (0, 1) {};
        \node[vertex] (14) at (0.5, 1) {};
        \node[vertex] (15) at (1, 1) {};
        \node[vertex] (16) at (1.5, 1) {};
        \node[vertex] (17) at (2, 1) {};
        \node[vertex] (18) at (2.5, 1) {};

        \node[vertex] (19) at (0, 1.5) {};
        \node[vertex] (20) at (0.5, 1.5) {};
        \node[vertex] (21) at (1, 1.5) {};
        \node[vertex] (22) at (1.5, 1.5) {};
        \node[vertex] (23) at (2, 1.5) {};
        \node[vertex] (24) at (2.5, 1.5) {};

        \node[vertex] (25) at (0, 2) {};
        \node[vertex] (26) at (0.5, 2) {};
        \node[vertex] (27) at (1, 2) {};
        \node[vertex] (28) at (1.5, 2) {};
        \node[vertex] (29) at (2, 2) {};
        \node[vertex] (30) at (2.5, 2) {};

        \node[vertex] (31) at (0, 2.5) {};
        \node[vertex] (32) at (0.5, 2.5) {};
        \node[vertex] (33) at (1, 2.5) {};
        \node[vertex] (34) at (1.5, 2.5) {};
        \node[vertex] (35) at (2, 2.5) {};
        \node[vertex] (36) at (2.5, 2.5) {};
        
        \draw (1) -- (2);
        \draw (2) -- (3);
        \draw (3) -- (4);
        \draw (4) -- (5);
        \draw (5) -- (6);

        \draw (7) -- (8);
        \draw (8) -- (9);
        \draw (9) -- (10);
        \draw (10) -- (11);
        \draw (11) -- (12);

        \draw (13) -- (14);
        \draw (14) -- (15);
        \draw (15) -- (16);
        \draw (16) -- (17);
        \draw (17) -- (18);

        \draw (19) -- (20);
        \draw (20) -- (21);
        \draw (21) -- (22);
        \draw (22) -- (23);
        \draw (23) -- (24);

        \draw (25) -- (26);
        \draw (26) -- (27);
        \draw (27) -- (28);
        \draw (28) -- (29);
        \draw (29) -- (30);

        \draw (31) -- (32);
        \draw (32) -- (33);
        \draw (33) -- (34);
        \draw (34) -- (35);
        \draw (35) -- (36);

        \draw (1) -- (7);
        \draw (2) -- (8);
        \draw (3) -- (9);
        \draw (4) -- (10);
        \draw (5) -- (11);
        \draw (6) -- (12);

        \draw (7) -- (13);
        \draw (8) -- (14);
        \draw (9) -- (15);
        \draw (10) -- (16);
        \draw (11) -- (17);
        \draw (12) -- (18);

        \draw (13) -- (19);
        \draw (14) -- (20);
        \draw (15) -- (21);
        \draw (16) -- (22);
        \draw (17) -- (23);
        \draw (18) -- (24);

        \draw (19) -- (25);
        \draw (20) -- (26);
        \draw (21) -- (27);
        \draw (22) -- (28);
        \draw (23) -- (29);
        \draw (24) -- (30);

        \draw (25) -- (31);
        \draw (26) -- (32);
        \draw (27) -- (33);
        \draw (28) -- (34);
        \draw (29) -- (35);
        \draw (30) -- (36);

        \draw    (1) .. controls (-0.25,1.25) and (-0.25,1.25) .. (31) ;
        \draw    (2) .. controls (0.25,1.25) and (0.25,1.25) .. (32) ;
        \draw    (3) .. controls (0.75,1.25) and (0.75,1.25) .. (33) ;
        \draw    (4) .. controls (1.25,1.25) and (1.25,1.25) .. (34) ;
        \draw    (5) .. controls (1.75,1.25) and (1.75,1.25) .. (35) ;
        \draw    (6) .. controls (2.25,1.25) and (2.25,1.25) .. (36) ;

        \draw (1) -- (36);
        \draw (7) -- (30);
        \draw (13) -- (24);
        \draw (19) -- (18);
        \draw (25) -- (12);
        \draw (31) -- (6);
        \end{tikzpicture}
    \end{subfigure}
    \caption{Illustration of a twisted torus graph on the left and a Klein bottle graph on the right}
    \label{twisted_torus_klein_bottle}
\end{figure}
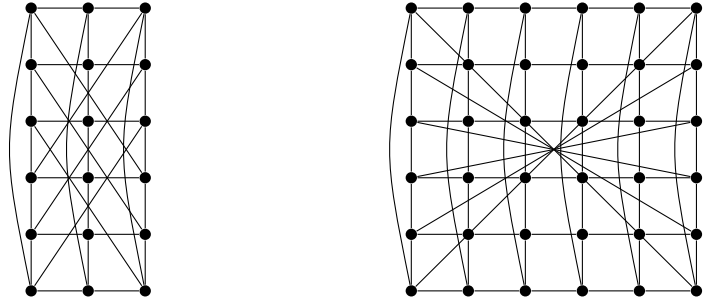

Therefore, one method for constructing 5-regular Ricci-flat graphs is to take the Cartesian product of a Ricci-flat 3-regular graph and a Ricci-flat 2-regular graph. By \thref{bone_idle_3_reg}, this approach is not applicable for constructing 5-regular bone-idle graphs, as no 3-regular bone-idle graphs exist.

The authors in \cite{Lei2022} study 5-regular Ricci-flat graphs. They are able to identify a 5-regular, symmetric graph of order 72, denoted $RF_{72}^{5}$, that is not of Cartesian product type.

\begin{theorem}[\cite{Lei2022}, Theorem 4.1]
    Let $G=(V,E)$ be a $5$-regular, symmetric graph.  If $G$ is Ricci-flat, then it is isomorphic to $RF_{72}^{5}$.
\end{theorem}

It is easy to verify that for every edge $x\sim y$ in $RF_{72}^{5}$, $\kappa_{0}(x,y) = -\frac{1}{5}$ holds. Therefore, we obtain the following result.

\begin{corollary}
    There exists no $5$-regular, symmetric graph that is bone-idle.
\end{corollary}

The authors also formulate the following conjecture. 

\begin{conjecture}[\cite{Lei2022}, Conjecture 1]
    If $G=(V,E)$ is a 5-regular Ricci-flat graph, then $G$ is either isomorphic to $RF_{72}^{5}$ or $G$ is of Cartesian product type.
\end{conjecture}

Should this conjecture prove true, no 5-regular bone-idle graphs exist. The existence of regular bone-idle graphs with an odd vertex degree remains an open question.

\bigskip

{\bf{Acknowledgements:}} I would like to express my deepest gratitude to Prof. Dr. Renesse and Dr. M{\"u}nch for their invaluable support and insightful feedback throughout the course of this work. 

\section*{Declarations}

{\bf{Funding:}} Partial financial support was received from the BMBF (Federal Ministry of Education and Research) in DAAD project 57616814 (SECAI, School of Embedded Composite AI).

{\bf{Conflict of interest:}} The author certifies that he has no affiliations with or involvement in any organization or entity with any financial interest or non-financial interest in the subject matter or materials discussed in this manuscript.

\end{document}